\documentclass[a4paper]{article}

\usepackage[utf8]{inputenc}
\usepackage{amsmath, amssymb, mathabx, mathtools,amsthm}
\usepackage{hyperref}
\usepackage{graphicx}
\usepackage{xcolor}
\usepackage[ruled]{algorithm2e}
\usepackage[hmarginratio=3:2]{geometry}
\usepackage[normalem]{ulem}
\newtheorem{definition}{Definition}[section]
\newtheorem{theorem}[definition]{Theorem}
\newtheorem{corollary}[definition]{Corollary}
\newtheorem{lemma}[definition]{Lemma}
\newtheorem{remark}[definition]{Remark}

\usepackage{chngcntr}
\counterwithin{figure}{section}
\counterwithin{table}{section}

\newcommand{\defas}{:=}
\newcommand{\ind}{\chi}

\newcommand{\Tc}{\mathcal{T}}
\newcommand{\Ec}{\mathcal{E}}
\newcommand{\Vc}{\mathcal{V}}
\newcommand{\Ic}{\mathcal{I}}
\newcommand{\Sc}{\mathcal{S}}
\newcommand{\neumannSpace}{\Vc}

\newcommand{\nlOp}{\mathcal{L}}
\newcommand{\neumannOp}{\mathcal{N}}
\newcommand{\neumannVarOp}{\mathcal{B}}
\newcommand{\nlDom}{\Omega}
\newcommand{\nlBound}{\Ic}
\newcommand{\extnlBound}{\hat{\nlBound}}
\newcommand{\completeDom}{\nlDom \cup \nlBound}
\newcommand{\varOp}{A}
\newcommand{\varForce}{F}

\newcommand{\xb}{\mathbf{x}}
\newcommand{\yb}{\mathbf{y}}
\newcommand{\zb}{\mathbf{z}}

\newcommand{\Nbb}{\mathbb{N}}

\newcommand{\kernel}{\gamma}

\newcommand{\kernelij}{\gamma_{ij}}

\newcommand{\kerneliI}{\gamma_{i \Ic}}
\newcommand{\kernelIi}{\gamma_{\Ic i}}

\newcommand{\mbR}{\mathbb{R}}
\newcommand{\mbRd}{\mathbb{R}^d}

\newcommand{\advar}{v}
\newcommand{\trialSpace}{V}
\newcommand{\testSpace}{V_c}

\newcommand{\basisfun}{\phi}
\newcommand{\FEMatrix}{\mathbf{A}}
\newcommand{\FEMatrixEntry}{\mathbf{a}}
\newcommand{\FEForce}{\mathbf{f}}
\newcommand{\FEForceAlt}{\mathbf{b}}
\newcommand{\FEForceEntry}{\mathbf{f}}
\newcommand{\FEDirData}{\mathbf{g}}
\newcommand{\FEu}{\mathbf{u}}

\newcommand{\indexSet}{\mathfrak{J}}

\newcommand{\projection}{\mathcal{P}}

\newcommand{\polynom}{p}
\newcommand{\multiIndex}{\boldsymbol{\alpha}}

\newcommand{\schwarzOp}{\Sc}
\newcommand{\schwarzSubspace}{\mathbb{V}}
\newcommand{\weakSol}{u}

\newcommand{\precBGS}{\mathbf{M}_{BGS}}
\newcommand{\precBJ}{\mathbf{M}_{BJ}}

\title{\textbf{Schwarz Methods for Nonlocal Problems}}
\author{Matthias Schuster\thanks{Universitaet Trier, D-54286 Trier, Germany; Email: schusterm@uni-trier.de, admin@christianvollmann.de, volker.schulz@uni-trier.de}\hspace{2mm}\href{https://orcid.org/0000-0002-9355-1076}{\includegraphics[scale=0.06]{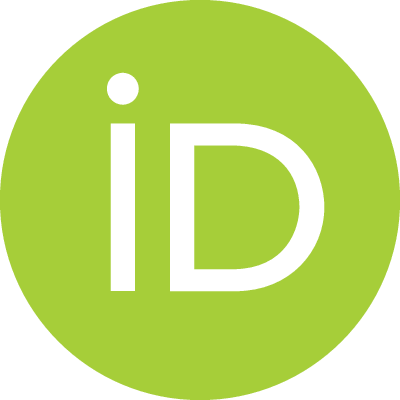}} \and Christian Vollmann\footnotemark[1]\hspace{2mm}\href{https://orcid.org/0000-0002-0351-0044}{\includegraphics[scale=0.06]{orcid.eps}} \and Volker Schulz\footnotemark[1]\hspace{2mm}\href{https://orcid.org/0000-0001-7665-130X}{\includegraphics[scale=0.06]{orcid.eps}} }
\date{}

\begin{document}
\maketitle
\small
\textbf{Abstract.}
The first domain decomposition methods for partial differential equations were already developed in 1870 by H. A. Schwarz. Here we consider a nonlocal Dirichlet problem with variable coefficients, where a nonlocal diffusion operator is used. We find that domain decomposition methods like the so-called Schwarz methods seem to be a natural way to solve these nonlocal problems. In this work we show the convergence for nonlocal problems, where specific symmetric kernels are employed, and present the implementation of the multiplicative and additive Schwarz algorithms in the above mentioned nonlocal setting.
~
\normalsize

\section{Introduction}
Nonlocal operators can be helpful to describe physical phenomena, where long range interactions occur. Therefore, in these cases, nonlocal models are better suited than the classical approach via partial differential equations. Nonlocal operators have already been applied in many fields, such as image denoising\cite{buades2010image, delia2021bilevel}, anomalous diffusion\cite{brockmann2008, suzuki2022, deliaAnomalousTransport}, peridynamics\cite{silling2000reformulation, javili2019peridynamics} and stochastic processes\cite{metzler2000random, delia2017nonlocal}, just to name a few.\\
Moreover, nonlocal interface problems as well as models that contain domain-dependent parameters have become popular in recent years, see e.g., \cite{shape_paper, capodaglio_energy_coupling, glusa2023asymptotically}. These domain-dependent parameters may naturally suggest a domain decomposition which leads to domain decomposition methods that are widely used in the case of partial differential equations and have been investigated in, e.g., \cite{mathewDD, toselli2004domain}.
In the matter of nonlocal equations the approach of finite element tearing and interconnecting (FETI) has been successfully applied as described in \cite{klar2023scalable, xu2021feti}. Another class of domain decomposition methods are Schwarz algorithms, that were initially developed in \cite{schwarz1869} for partial differential equations. An abstract framework for the Schwarz approach can be found in \cite{LionsSchwarz1, LionsSchwarz2, LionsSchwarz3}.\\
Although FETI is a very efficient way to solve nonlocal equations, Schwarz methods have several advantages. They are easier to implement and it is straightforward to incorporate black-box solvers, whereas for FETI the finite element formulation on overlapping domains needs to be changed. Moreover, we could observe that Schwarz methods can also converge in case of a nonsymmetric kernel as shown in an experiment in Chapter \ref{chap:patch_test_experiments}. Lastly, if the parameter $\delta$, which determines the the range of nonlocal interaction and thus the width of the nonlocal boundary regarding one (sub-)domain (see Section \ref{chap:problem_formulation}), is quite huge, FETI cannot decompose the domain in many subdomains. In this case, GMRES with a Schwarz preconditioner might be faster, especially if we include the set-up times of both solvers. \\
Schwarz methods have also been utilized to solve an energy-based Local-to-Nonlocal(LtN) coupling in \cite{AcostaDD}. Additionally, in \cite{AcostaDD} they showed the convergence of the multiplicative Schwarz method for this energy-based LtN coupling by applying \cite[Theorem I.1]{LionsSchwarz1}. In Section \ref{sec:well_posedness} we present the convergence of the multiplicative Schwarz method for two cases of nonlocal problems in a similar fashion by referring to \cite[Theorem I.2]{LionsSchwarz1}. Furthermore, we point out that the Schwarz formulation solves linear, quadratic and cubic patch tests. These patch tests often serve as a criterion to evaluate and therefore compare Local-to-Nonlocal couplings. We refer to the review paper \cite{coupling_review} for more information on LtN couplings. Additionally we test GMRES against two preconditioned versions of GMRES that result from the multiplicative and additive Schwarz method combined with the finite element method.\\
We start by introducing nonlocal Dirichlet problems and nonlocal problems that involve Neumann boundary conditions in Chapter \ref{chap:problem_formulation}. After that we describe the multiplicative and additive Schwarz method for the previously mentioned nonlocal problems in Chapter \ref{chap:schwarz_methods}. Next, in Chapter \ref{sec:well_posedness}, we prove the convergence of this multiplicative Schwarz method to the unique solution under the assumption that the kernel is symmetric. Furthermore, we present the finite element method as a way to numerically solve nonlocal problems in Chapter \ref{chap:finite_element_method}.
Here, we also illustrate how Schwarz methods are applied for this discrete framework. In Chapter \ref{chap:patch_tests} we shortly introduce patch tests, which are used to evaluate Local-to-Nonlocal or Nonlocal-to-Nonlocal couplings. In the last Chapter \ref{chap:numerical_experiments} we start by testing the Schwarz methods on two examples. The first one is a nonlocal Dirichlet problem, where we use a singular symmetric kernel, and the second example is a nonlocal problem with Neumann boundary, where we employ a constant kernel. Then, we investigate an example of a nonlocal Dirichlet problem with a nonsymmetric kernel that additionally satisfies the patch test. Lastly, we study two versions of preconditioned GMRES, that are connected to the Schwarz method for the finite element formulation, and we compare these two approaches to a GMRES without preconditioning.
 
\section{Problem Formulation}\label{chap:problem_formulation}
Let $\nlDom \subset \mbRd$ be an open and bounded domain and $\{\Omega_i\}_{i=1}^n$ be a partition of $\Omega$, where $\nlDom_i \subset \nlDom$ is open for $i=1,...,n$, and that satisfies
\begin{align} \label{decomposition}
    \overline{\nlDom} = \bigcup_{i=1}^n \overline{\Omega}_i \text{ and } \Omega_i \cap \Omega_j = \emptyset \text{ for } i,j = 1,...,n \text{ and } i \neq j.
\end{align}
From now on we denote by $\kernel:\mbRd \times \mbRd \rightarrow [0,\infty)$ a nonnegative (interaction) kernel and define the nonlocal boundary $\nlBound$ of $\Omega$ as follows 
\begin{align*}
 \nlBound \defas \{\yb \in \mbRd \setminus \nlDom: \int_{\nlDom} \kernel(\xb, \yb) ~d\xb > 0 \}.   
\end{align*}
Based on the partition $\{\Omega_i\}_{i=1}^n$ we can express the kernel $\kernel$ as  
\begin{align}\label{kernel_sum}
    \kernel = \sum_{i,j=1}^n \kernelij \ind_{\Omega_i \times \Omega_j} + \sum_{i=1}^n \kerneliI \ind_{\Omega_i \times \nlBound} + \sum_{i=1}^n \kernelIi \ind_{\nlBound \times \Omega_i},
\end{align}
with potentially different nonnegative kernels 
\begin{align*}
\kernelij:\Omega_i \times \Omega_j \rightarrow [0,\infty), \kerneliI:\Omega_i \times \nlBound \rightarrow [0,\infty) \text{ and } \kernelIi:\nlBound \times \Omega_i \rightarrow [0,\infty), \text{ for } i,j=1,...,n.
\end{align*}
Then, we define the nonlocal diffusion operator $-\nlOp$, which can be found in, e.g., \cite{delia2017nonlocal, delia_glusa_interface, DuAnalysis, vollmann_diss, vu_diss}, as
\begin{align}
\label{op:conv-diff}
    -\nlOp u(\xb) = \int\limits_{\mbRd} \left(u(\xb) - u(\yb) \right)\kernel(\xb,\yb) ~d\yb.
\end{align}
\subsection{Nonlocal Dirichlet Problems}
\label{sec:dirichlet_problem}
One class of problems, that we focus on, are nonhomogeneous steady-state Dirichlet problems with nonlocal boundary constraints given by
\begin{align}
	\begin{split}
		\label{nlProblem}
		-\nlOp u &= f \text{ on } \nlDom\\
		u &= g \text{ on } \nlBound, 
	\end{split}
\end{align}
where the forcing term $f$ may vary across the various parts of the decomposition \eqref{decomposition}, i.e., there exist possibly mutually different $f_i:\Omega_i \rightarrow \mbR, \text{ for } i=1,...,n,$ such that
\begin{align} \label{def:forcing_term}
    f = \sum_{i=1}^n f_i \ind_{\Omega_i}.
\end{align}
Similar types of nonlocal problems can also be found in, e.g., \cite{delia_glusa_interface, shape_paper}. Moreover, the kernel $\kernel$ is assumed to be symmetric and is supposed to fulfill the following conditions
\begin{itemize}
	\item[(K1)] There exist $\delta > 0$ and $\phi:\mbRd \times \mbRd \rightarrow [0,\infty )$ with $\kernel(\xb,\yb) = \phi(\xb,\yb) \ind_{B_\delta(\xb)}(\yb)$.
	\item[(K2)] There exist $0 <\kernel_0 < \infty$ and $\epsilon \in (0, \delta)$ such that $\kernel_0 \leq \kernel(\xb,\yb)$ for $\yb \in B_{\epsilon}(\xb)$ and $\xb \in \nlDom$.
\end{itemize}
In order to derive a weak formulation, we multiply the first equation of the left-hand side of \eqref{nlProblem} by a test function $\advar: \nlDom \cup \nlBound \rightarrow \mbR$ and then integrating over $\nlDom$ yields
\begin{align}
    \left(-\nlOp u, \advar \right)_{L^2(\nlDom)} &= \left(f, \advar \right)_{L^2(\nlDom)} \nonumber \\
    \Leftrightarrow \int_{\nlDom} \advar(\xb)\int_{\mbRd} (u(\xb) &- u(\yb)) \kernel(\xb,\yb) ~d\yb d\xb = \int_{\nlDom} f(\xb) \advar(\xb) ~d\xb. \label{firstVarForm}
\end{align}
By applying Fubini’s theorem and by considering $\advar(\yb) = 0$ on $\nlBound$ the first integral of the left-hand side can be expressed as
\begin{align*}
    &\int_{\nlDom} \advar(\xb)\int_{\completeDom} (u(\xb) - u(\yb))\kernel(\xb,\yb) ~d\yb d\xb\\ 
    &= \frac{1}{2} \iint_{(\completeDom)^2} (\advar(\xb) - \advar(\yb)) (u(\xb) - u(\yb))\kernel(\xb,\yb) ~d\yb d\xb.
\end{align*}
Thus we now define the nonlocal bilinear form
\begin{align*}
    \varOp(u,v) &\defas \frac{1}{2} \iint_{(\completeDom)^2} (\advar(\xb) - \advar(\yb)) (u(\xb) - u(\yb))\kernel(\xb,\yb) ~d\yb d\xb\\
    &=\sum_{i,j= 1}^n \frac{1}{2}\int_{\Omega_i} \int_{\Omega_j}\left(\advar(\xb) - \advar(\yb) \right) \left(u(\xb) - u(\yb)\right)\kernelij(\xb,\yb) ~d\yb d\xb \\
    &~~~~+ \sum_{i=1}^n \int_{\Omega_i}  \int_{\nlBound} (\advar(\xb) - \advar(\yb) ) (u(\xb) - u(\yb))\kerneliI(\xb,\yb) ~d\yb d\xb 
  \end{align*}  
  and the linear functional
\begin{align*}
    \varForce(v) &\defas \int_{\nlDom} f(\xb) \advar(\xb) ~d\xb = \sum_{i=1}^n \int_{\Omega_i} f_i(\xb) \advar(\xb) ~d\xb,
\end{align*}
where we used the definition of the forcing term \eqref{def:forcing_term} and that the kernel can be expressed as a sum of kernels as in \eqref{kernel_sum}.
Moreover, with $\varOp$ we can define the following semi-norm and norm, respectively,
\begin{align}
    \begin{split} \label{norm}
        |u|_{\trialSpace(\completeDom)} &\defas \sqrt{\varOp(u,u)}, \\ 
    ||u||_{\trialSpace(\completeDom)} &\defas |u|_{\trialSpace(\completeDom)} + ||u||_{L^2(\completeDom)}.
    \end{split}
\end{align}
This enables us to define the following spaces
\begin{align*}
    \trialSpace(\completeDom) &\defas \{u \in L^2(\completeDom): ||u||_{\trialSpace(\completeDom)} < \infty\} \text{ and} \\
    \testSpace(\completeDom) &\defas \{u \in \trialSpace(\completeDom): u = 0 \text{ on } \nlBound\}.
\end{align*}\\
We now present a variational formulation of problem \eqref{nlProblem}:
\begin{definition}
	Given $f \in L^2(\nlDom)$ and $g \in \trialSpace(\completeDom)$, if $\weakSol \in \trialSpace(\completeDom)$ solves
	\begin{align}\label{weak_formulation}
		\begin{split}
		\varOp(\weakSol, \advar) &= \varForce(\advar) \text{ for all } \advar \in \testSpace(\completeDom) \text{ and} \\
		u&-g \in \testSpace(\completeDom),
	\end{split}	
\end{align}
	then $u$ is called weak solution of \eqref{nlProblem}.
\end{definition}~\\
Since $\varOp(u,\advar) = \varOp(u - g,\advar) + \varOp(g,\advar)$ and $u-g \in \testSpace(\completeDom)$,
problem \eqref{weak_formulation} can also be reformulated as a homogeneous Dirichlet problem
\begin{align}
	\label{weak_formulation_homogeneous}
	\begin{split}
	&\textit{Given $f \in L^2(\nlDom)$ and $g \in \trialSpace(\completeDom)$, find } \tilde{\weakSol} \in \testSpace(\completeDom):\\
	&\varOp(\tilde{\weakSol}, \advar) = \varForce(\advar) - \varOp(g,\advar) \textit{ for all } \advar \in \testSpace(\completeDom),\\
	&\textit{then } \weakSol \defas \tilde{\weakSol} + g \in \trialSpace(\completeDom) \textit{ is a weak solution of \eqref{weak_formulation}.}
\end{split}
\end{align} 
\begin{remark}
	\label{remark:examples_well_posedness}
	In this work, we show the convergence of the multiplicative Schwarz method for nonlocal Dirichlet problems for symmetric kernels, that fulfill (K1) and (K2). Additionally, the kernels are assumed to satisfy the requirements of one of the following two classes, where the well-posedness of the corresponding nonlocal Dirichlet problem \eqref{weak_formulation} is shown, e.g., in \cite{DuAnalysis}:
	\begin{itemize}
		\item Integrable kernels:\\
		There exist constants $0 < \kernel_1, \kernel_2 < \infty$ with
		\begin{align*}
			\kernel_1 \leq \inf_{\xb \in \nlDom} \int_{\completeDom} \kernel(\xb,\yb) ~d\yb \text{ and } \sup_{\xb \in \nlDom} \int_{\completeDom} \left( \kernel(\xb,\yb) \right)^2 ~d\yb \leq \left( \kernel_2 \right)^2.
		\end{align*}
	Then the solution $u \in \trialSpace(\completeDom)$ for \eqref{weak_formulation} exists, is unique and in this case the spaces $\left(\testSpace(\completeDom), |\cdot|_{\trialSpace(\completeDom)} \right)$ and $\left( L_c^2(\completeDom), ||\cdot||_{L^2(\completeDom)} \right)$ as well as $\left(\trialSpace(\completeDom), ||\cdot||_{\trialSpace(\completeDom)} \right)$ and $\left( L^2(\completeDom), ||\cdot||_{L^2(\completeDom)} \right)$ are equivalent.
		\item Singular kernels:\\
		There exist constants $0 < \kernel_*, \kernel^* < \infty$ and $s \in (0,1)$ with
		\begin{align*}
			\kernel_* < \kernel(\xb,\yb)||\xb + \yb||_2^{d+2s} < \kernel^* \text{ for all } \yb \in B_{\delta}(\xb) \text{ and } \xb \in \nlDom.
		\end{align*}
	Then the solution to \eqref{weak_formulation} exists, is unique and the spaces $\left(\testSpace(\completeDom), |\cdot|_{\trialSpace(\completeDom)} \right)$ and \\ $\left( H^s_c(\completeDom), |\cdot|_{H^s(\completeDom)} \right)$ as well as $\left(\trialSpace(\completeDom), ||\cdot||_{\trialSpace(\completeDom)} \right)$ and $\left(H^s(\completeDom), ||\cdot||_{H^s(\completeDom)} \right)$ are equivalent.
	\end{itemize}
\end{remark}
\subsection{Nonlocal Problems with Robin Boundary Conditions}
\label{chap:nl_prob_robin}
Besides solving a nonlocal Dirichlet problem, we are also interested in solving nonlocal problems that involve a Neumann boundary condition, i.e., find a solution $u$ for
\begin{align}
    \label{nlProb_RobinBoundary}
    \begin{split}
        - \nlOp u(\xb) + \kappa(\xb)u(\xb) &= f(\xb)  \text{ for } \xb \in \nlDom, \\
    \neumannOp u(\yb) &= g^N(\yb)  \text{ for } \yb \in \nlBound^N, \\
    u(\yb) &= g^D(\yb) \text{ for } \yb \in \nlBound^D,
    \end{split}
\end{align}
where
\begin{align}
	\nlBound = \nlBound^N \dot{\cup} \nlBound^D,\ \kappa&:\nlDom \rightarrow [\alpha, \beta] \text{ with } 0 < \alpha \leq \beta < \infty,\ \kernel: \mbRd \times \mbRd \rightarrow [0, \infty ) \text{ measurable and} \nonumber \\
\neumannOp(\yb) &\defas \int_{\completeDom} \left( u(\xb) - u(\yb) \right)\kernel(\xb, \yb) ~d\yb \text{ for } \yb \in \nlBound^N. \label{neumann_op}
\end{align}
Here, $\neumannOp$ is called nonlocal Neumann operator. For problems of type \eqref{nlProb_RobinBoundary}, we only assume $\kernel:\mbRd \times \mbRd \rightarrow[0,\infty)$ to be a measurable and symmetric function.
Furthermore, if $\kernel(\xb,\yb) = 0$ for $\xb,\yb \in \nlBound$, the Neumann operator reduces to
\begin{align}
	\label{vu_neumann_op}
	\neumannOp(\yb) = \int_{\nlDom} \left( u(\xb) - u(\yb) \right)\kernel(\xb, \yb) ~d\yb \text{ for } \yb \in \nlBound^N.
\end{align}
In the case of $\nlBound^N=\nlBound$ and \eqref{vu_neumann_op} there exists a unique weak solution to problem \eqref{nlProb_RobinBoundary} according to \cite{vu_diss}. In the following we show existence of a unique solution to problem \eqref{nlProb_RobinBoundary} for the more general case \eqref{nlProb_RobinBoundary} by following the approach in \cite{vu_diss}.
Therefore, in order to get a variational formulation for the problem \eqref{nlProb_RobinBoundary}, we multiply the first two equations of \eqref{nlProb_RobinBoundary} by a test function $\advar: \nlDom \cup \nlBound \rightarrow \mbR$, integrate over $\nlDom$, or $\nlBound^N$ respectively, and add up the two equations. Then we get
\begin{align}
	\label{neumann_var_Form_start}
	\left( - \nlOp u + \kappa u, \advar \right)_{L^2(\nlDom)} + \left( \neumannOp u, \advar \right)_{L^2(\nlBound^N)} = \left(f, \advar \right)_{L^2(\nlDom)} + \left(g^N, \advar \right)_{L^2(\nlBound^N)}.
\end{align}
By assuming $\advar = 0$ on $\nlBound^D$ and by using similar computations as in Section \ref{chap:problem_formulation} we can reformulate \eqref{neumann_var_Form_start} as
\begin{align}
	\begin{split}\label{neumann_var_Form_end}
	&\frac{1}{2} \int_{\completeDom} \int_{\completeDom} \left( \advar(\xb) - \advar(\yb) \right) \left( u(\xb) - u(\yb) \right) \kernel(\xb,\yb) ~d\yb d\xb
	+ \int_{\nlDom} \kappa(\xb) \advar(\xb) u(\xb) ~d\xb \\
	&= \int_{\nlDom} f(\xb) \advar(\xb) ~d\xb + \int_{\nlBound^N} g^N(\xb) \advar(\xb) ~d\xb.
\end{split}
\end{align}
Next, we define the spaces
\begin{alignat*}{2}
    \neumannSpace(\nlDom, \nlBound^N, \nlBound^D) &\defas \{\advar \in L^2(\completeDom): ||\advar ||_{\neumannSpace(\nlDom, \nlBound^N, \nlBound^D)} < \infty \}, \text{ and} \\
    \neumannSpace_c(\nlDom, \nlBound^N, \nlBound^D) &\defas \{\advar \in \neumannSpace(\nlDom, \nlBound^N, \nlBound^D) \text{ with } \advar = 0 \text{ on } \nlBound^D \}, \text{ where} \\
    ||\advar ||^2_{\neumannSpace(\nlDom, \nlBound^N, \nlBound^D)} &\defas \int_\nlDom \advar(\xb)^2 ~d\xb + \int_{\completeDom} \int_{\completeDom} \left( \advar(\xb) - \advar(\yb) \right)^2\kernel(\xb,\yb) ~d\yb d\xb.
\end{alignat*}
Moreover, set
\begin{align*}
	\langle u,\advar \rangle_{\neumannSpace(\nlDom, \nlBound^N, \nlBound^D)} \defas \int_\nlDom u(\xb)\advar(\xb) ~d\xb &+ \int_{\completeDom} \int_{\completeDom} \left( u(\xb) - u(\yb) \right) \left( \advar(\xb) - \advar(\yb) \right)\kernel(\xb,\yb) ~d\yb d\xb,
\end{align*}
which is a semi-inner product for $\neumannSpace_c(\nlDom, \nlBound^N, \nlBound^D)$. We further define the quotient space
\begin{align*}
	\trialSpace(\nlDom, \nlBound^N, \nlBound^D) &\defas \neumannSpace(\nlDom, \nlBound^N, \nlBound^D)/Z, \text{ and }
	\testSpace(\nlDom, \nlBound^N, \nlBound^D) \defas \neumannSpace_c(\nlDom, \nlBound^N, \nlBound^D)/Z, \text{ where} \\
	Z &\defas \{ \advar \in \neumannSpace_c(\nlDom, \nlBound^N, \nlBound^D): \advar(\xb) = 0 \text{ for a.e. } \xb \in \completeDom \} \text{ and set}\\
	\langle [u], [\advar] \rangle_{\trialSpace(\nlDom, \nlBound^N, \nlBound^D)} &\defas \langle u,\advar \rangle_{\neumannSpace(\nlDom, \nlBound^N, \nlBound^D)} \text{ and } ||[\advar]||_{\trialSpace(\nlDom, \nlBound^N, \nlBound^D)} = ||\advar||_{\neumannSpace(\nlDom, \nlBound^N, \nlBound^D)}.
\end{align*}
In the following we write $u$ instead of $[u]$ for ease of presentation.\\
The first line of \eqref{neumann_var_Form_end} yields the bilinear form $\neumannVarOp:\trialSpace(\nlDom, \nlBound^N, \nlBound^D) \times \trialSpace(\nlDom, \nlBound^N, \nlBound^D) \rightarrow \mbR,$ where
\begin{align*}
    \neumannVarOp(u,\advar) \defas \frac{1}{2} \int\limits_{\completeDom} \int\limits_{\completeDom} \left(u(\xb) - u(\yb) \right) \left( \advar(\xb) - \advar(\yb) \right) \kernel(\xb,\yb) ~d\yb d\xb.
\end{align*}
Additionally, we define the linear functional $\varForce^N:\trialSpace(\nlDom, \nlBound^N, \nlBound^D) \rightarrow \mbR$ as
\begin{align*}
	\varForce^N(\advar) = \int_{\nlDom} f(\xb)\advar(\xb) ~d\xb + \int_{\nlBound^N} g^N(\xb)\advar(\xb) ~d\xb.
\end{align*}
\begin{definition}
	Given functions $g^D \in \trialSpace(\nlDom, \nlBound^N, \nlBound^D)$, $f \in L^2(\nlDom)$, $g^N \in L^2(\nlBound^N)$ and ${\kappa:\nlDom \rightarrow [\alpha,\beta]}$, where $0 < \alpha \leq \beta < \infty$, then, if a function $\weakSol \in \trialSpace(\nlDom, \nlBound^N, \nlBound^D)$ solves
	\begin{align}
		\label{neumannVarFormulation}
		\begin{split}
		\neumannVarOp(u,\advar) + \int_{\nlDom} \kappa(\xb) u(\xb) \advar(\xb) ~d\xb &= \varForce^N(\advar) \text{ for all } \advar \in \testSpace(\nlDom, \nlBound^N, \nlBound^D) \text{ and} \\
		\weakSol - g^D &\in \testSpace(\nlDom, \nlBound^N, \nlBound^D),
	\end{split}
	\end{align}
	the function $\weakSol$ is called weak solution to problem \eqref{nlProb_RobinBoundary}.
\end{definition}~\\
Again, this weak formulation can also be reformulated as a problem with homogeneous Dirichlet boundary conditions:
\begin{align}
	&\textit{Given functions $g^D \in \trialSpace(\nlDom, \nlBound^N, \nlBound^D)$, $f \in L^2(\nlDom)$, $g^N \in L^2(\nlBound^N)$ and $\kappa:\nlDom \rightarrow [\alpha,\beta]$,}\nonumber\\ 
    &\textit{where $0 < \alpha \leq \beta < \infty$, find a function } \tilde{\weakSol} \in \testSpace(\nlDom, \nlBound^N, \nlBound^D) \textit{ such that} \nonumber \\
		&\neumannVarOp(\tilde{\weakSol},\advar) + \int_{\nlDom} \kappa(\xb) \tilde{\weakSol}(\xb) \advar(\xb) ~d\xb = \varForce^N(\advar)
		- \left(\neumannVarOp(g^D,\advar)
		+ \int_{\nlDom} \kappa(\xb) g^D(\xb) \advar(\xb) ~d\xb \right) \label{neumannVarFormulation_homogeneous} \\ 
		&\textit{ for all } \advar \in \testSpace(\nlDom, \nlBound^N, \nlBound^D), \textit{then } \weakSol \defas \tilde{\weakSol} + g^D \in \trialSpace(\nlDom, \nlBound^N, \nlBound^D) \textit{ solves \eqref{neumannVarFormulation}.} \nonumber
\end{align}
Now, we can show the following theorem, which is similar to \cite[Theorem 3.1]{vu_diss}.
\begin{theorem}
	\label{theorem:neumann_hilbert}
	The space $\neumannSpace_c(\nlDom, \nlBound^N, \nlBound^D)$ is complete regarding $||\cdot||_{\neumannSpace(\nlDom, \nlBound^N, \nlBound^D)}$. As a consequence, $\testSpace(\nlDom, \nlBound^N, \nlBound^D)$ is a Hilbert space with respect to $||\cdot||_{\trialSpace(\nlDom, \nlBound^N, \nlBound^D)}$. 
\end{theorem}
\begin{proof}
	Let $(\advar_n)_{n \in \Nbb}$ be a Cauchy sequence in $\neumannSpace_c(\nlDom, \nlBound^N, \nlBound^D)$. Then, due to the completeness of $L^2(\nlDom)$, respectively $L^2((\completeDom^N) \times (\completeDom^N))$, there exist functions $\tilde{\advar} \in L^2(\nlDom)$ and $w \in L^2((\completeDom^N) \times (\completeDom^N))$, such that
	\begin{align*}
		&\lim_{k \rightarrow \infty} \int_{\nlDom} \left( \advar_k(\xb) - \tilde{\advar}(\xb) \right)^2 ~d\xb = 0 \text{ and} \\
		&\lim_{k \rightarrow \infty} \int_{\completeDom^N} \int_{\completeDom^N} \left( (\advar_k(\xb) - \advar_k(\yb))\sqrt{\kernel(\xb,\yb)} - w(\xb,\yb) \right)^2 ~d\yb d\xb = 0.
	\end{align*}
Analogously to \cite[Theorem 3.1]{vu_diss} one can show that there exists a subsequence $\left( \advar_{n_l} \right)_{l \in \Nbb}$ of $(\advar_n)_{n \in \Nbb}$ and an extension $\advar:\completeDom \rightarrow \mbR$ of  $\tilde{\advar}$, i.e., $\advar = \tilde{\advar}$ on $\nlDom$, with $\advar = 0$ on $\nlBound^D$ and
\begin{align*}
	\lim_{l \rightarrow \infty} \advar_{n_l}(\xb) = \advar(\xb) \text{ for a.e. } \xb \in \completeDom.
\end{align*}
Consequently, we derive
\begin{align*}
	w(\xb,\yb) = \lim_{l \rightarrow \infty} \left( \advar_{n_l}(\xb) - \advar_{n_l}(\yb) \right) \sqrt{\kernel(\xb,\yb)} = \left( \advar(\xb) - \advar(\yb) \right) \sqrt{\kernel(\xb,\yb)} \text{ for a.e. } (\xb,\yb) \in (\completeDom)^2.
\end{align*}
Therefore, $\advar \in \neumannSpace_c(\nlDom, \nlBound^N, \nlBound^D)$ and we can conclude
\begin{align*}
	\lim_{l \rightarrow \infty} ||\advar_{n_l} - \advar||_{\neumannSpace(\nlDom, \nlBound^N, \nlBound^D)} = 0 \text{ and } \lim_{n \rightarrow \infty} ||\advar_{n} - \advar||_{\neumannSpace(\nlDom, \nlBound^N, \nlBound^D)} = 0.
\end{align*}
\end{proof}~\\
Set $|||u|||_{\trialSpace(\nlDom, \nlBound^N, \nlBound^D)}^2\defas \neumannVarOp(u, u) + \int\limits_{\nlDom} \kappa(\xb)u^2(\xb) ~d\xb$, then the norms $|||\cdot|||_{\trialSpace(\nlDom, \nlBound^N, \nlBound^D)}$ and ${||\cdot||_{\trialSpace(\nlDom, \nlBound^N, \nlBound^D)}}$ are equivalent for $u \in \testSpace(\nlDom, \nlBound^N, \nlBound^D)$ since 
\begin{align}\label{ineq:neumann_norms}
	\min\{\frac{1}{2}, \alpha \} ||u||^2_{\trialSpace(\nlDom, \nlBound^N, \nlBound^D)} \leq |||u|||_{\trialSpace(\nlDom, \nlBound^N, \nlBound^D)}^2 \leq \max\{ \frac{1}{2}, \beta \}||u||^2_{\trialSpace(\nlDom, \nlBound^N, \nlBound^D)}.
\end{align}
By applying the Riesz representation theorem we directly get the following conclusion. 
\begin{corollary}
	Given $f \in L^2(\nlDom)$, $g^N \in L^2(\nlBound^N)$ and $g^D \in \trialSpace(\nlDom, \nlBound^N, \nlBound^D)$, there exists a unique function $\tilde{\weakSol} \in \testSpace(\nlDom, \nlBound^N, \nlBound^D)$ that satisfies \eqref{neumannVarFormulation_homogeneous}, i.e., $\weakSol = \tilde{\weakSol} + g^D$ is a weak solution for \eqref{nlProb_RobinBoundary}.
\end{corollary}

\section{Schwarz Methods}
\label{chap:schwarz_methods}
The first domain decomposition method was already formulated by Hermann Amandus Schwarz around 1869 and is now known as the Schwarz alternating method\cite{schwarz1869, schwarz1870}. It is an iterative method which was originally used to solve the Laplace equation on a domain by decomposing the domain into two overlapping domains where the Laplace equation could be easily solved on each. In every iteration the former solution regarding the problem on the other domain was used as a boundary condition on the boundary that was part of the other domain.\\ 
In this section we formulate the additive and multiplicative Schwarz method for the nonlocal problem \eqref{nlProblem} in the same manner. Therefore, we denote the nonlocal boundary regarding the subdomains $\Omega_i$ of the decomposition \eqref{decomposition} by
\begin{align*}
    \nlBound_i \defas \{\yb \in \mbRd \setminus \Omega_i: \int_{\Omega_i} \kernel(\xb, \yb) d\xb > 0 \}.
\end{align*}
\subsection{Schwarz Methods for Nonlocal Dirichlet Problems}
\label{chap:schwarz_dirichlet}
Given the solution $u$ of the nonlocal problem \eqref{weak_formulation}, for $i=1,...,n$ we set $u_i \defas u|_{\Omega_i \cup \nlBound_i},$ and observe that $u_i$ is a solution of
\begin{align}
    \textit{Find a weak solution } \weakSol_i &\in \trialSpace(\nlDom_i \cup \nlBound_i) \textit{ subject to } \nonumber \\
    \begin{split} \label{schwarz_formulation}
        -\nlOp u_i &= f_i \textit{ on } \nlDom_i, \\
           u_i &= u_j \textit{ on } \nlBound_i \cap \nlDom_j, \textit{ for } j = 1,...,n   \textit{ and } j\neq i, \\
           u_i &= g \textit{ on } \nlBound_i \cap \nlBound.
    \end{split}
\end{align}
In this case, $\weakSol|_{\nlDom_j \cup \nlBound_j} \in \trialSpace(\nlDom_j \cup \nlBound_j)$ for $j=1,...,n$ and $j \neq i$ serve as the boundary data.\\
On the other hand, as we will show in Section \ref{well-posedness_dirichlet}, if we have weak solutions $\{u_i\}_{i=1}^n$ for \eqref{schwarz_formulation}, the function $$u \defas \sum_{i=1}^n u_i \ind_{\nlDom_i} + g \ind_{\nlBound}$$ solves the nonlocal problem \eqref{nlProblem}.
\begin{remark}
	\label{remark:boundary_extension}
	Since we derive \eqref{schwarz_formulation} from \eqref{nlProblem} and \eqref{weak_formulation}, respectively, the boundary data of the problem on $\nlDom_i$ is an element of $\trialSpace(\nlDom_i \cup \nlBound_i)$ and can naturally be extended to a function of $\trialSpace(\completeDom)$. Therefore, we expand the boundary $\nlBound_i$ to $\hat{\nlBound}_i \defas (\nlDom \setminus \nlDom_i) \cup \nlBound$ and restrict ourselves to boundary data in $\trialSpace(\nlDom_i \cup \hat{\nlBound}_i)$. Consequently, we will make use of the spaces $\trialSpace(\nlDom_i \cup \hat{\nlBound}_i) = \trialSpace(\completeDom)$ and $\testSpace(\nlDom_i \cup \hat{\nlBound}_i) = \{\weakSol \in \trialSpace(\completeDom): \weakSol = 0 \text{ on } \hat{\nlBound}_i =\left( \nlDom \setminus \nlDom_i \right) \cup \nlBound  \}$ in order to formulate a weak formulation for \eqref{schwarz_formulation}.
\end{remark}~\\
Thus, for $\advar_i \in \testSpace(\nlDom_i \cup \hat{\nlBound}_i)$ and $\weakSol \in \trialSpace(\completeDom)$ the bilinear form $\varOp$ reduces to the corresponding bilinear operator that only contains integrals over $\left( \nlDom_i \cup \nlBound_i \right) \times \left( \nlDom_i \cup \nlBound_i \right)$ as follows
\begin{align*}
    \varOp(\weakSol, \advar_i) = &\sum_{l,j= 1}^n \frac{1}{2} \int_{\nlDom_l} \int_{\nlDom_j} \left( \weakSol(\xb) - \weakSol(\yb) \right) \left( \advar_i(\xb) - \advar_i(\yb) \right) \kernel(\xb,\yb) ~d\yb d\xb \\
	&+ \sum_{j=1}^n \int_{\nlDom_j} \weakSol(\xb) \advar_i(\xb) \int_{\nlBound} \kernel(\xb,\yb) ~d\yb d\xb \\
	&= \frac{1}{2} \int_{\nlDom_i} \int_{\nlDom_i} \left( \weakSol(\xb) - \weakSol(\yb) \right) \left( \advar_i(\xb) - \advar_i(\yb) \right) \kernel(\xb,\yb) ~d\yb d\xb \\
	&+ \sum_{\substack{j=1 \\ j \neq i}}^n \int_{\nlDom_i} \int_{\nlDom_j} \advar_i(\xb) \left( \weakSol(\xb) - \weakSol(\yb) \right) \kernel(\xb, \yb) ~d\yb d\xb + \int_{\nlDom_i} \weakSol(\xb) \advar_i(\xb) \int_{\nlBound} \kernel(\xb,\yb) ~d\yb d\xb,
\end{align*}
where we used the symmetry of $\kernel$ in the last step. Analogously, $\varForce$ can in this case be written as an operator that only integrates over $\nlDom_i$, i.e.,
\begin{align*}
	\varForce(\advar_i) = &\int_{\nlDom} f\advar_i ~d\xb = \int_{\nlDom_i} f_i\advar_i ~d\xb. 
\end{align*}
Then, we define the variational formulation for the subproblem on $\nlDom_i$ as:
\begin{align}
	\label{weak_formulation_subdomain}
	\begin{split}
        \textit{Given } f \in L^2(\nlDom) \textit{ and } &\textit{boundary data } g \in \trialSpace(\nlDom_i \cup \extnlBound_i)=\trialSpace(\completeDom),\\
        \textit{ find } \weakSol_i \in \trialSpace(\nlDom_i \cup \extnlBound_i) &=\trialSpace(\completeDom) \textit{ such that }\\
		\varOp(\weakSol_i, \advar_i) &= \varForce(\advar_i) \textit{ for all } \advar_i \in \testSpace(\nlDom_i \cup \hat{\nlBound}_i), \\
		\weakSol_i - g &\in \testSpace(\nlDom_i \cup \hat{\nlBound}_i).
	\end{split}
\end{align}
In the following we show two popular Schwarz methods to solve the problem \eqref{schwarz_formulation} in an iterative manner. In both cases we start with initial guesses $\{u_i^0\}_{i=1}^n$.
In the multiplicative Schwarz method, as shown in Algorithm \ref{alg:multiplicative_schwarz}, we use the most recent solutions in every iteration as the boundary data, i.e., for the subproblem $i$ the solutions $u_j^{k+1}$ for $0 \leq j < i$ from the current iteration are employed. \\
\begin{algorithm}[H]
\label{alg:multiplicative_schwarz}
\SetAlgoLined
\DontPrintSemicolon
\SetKwInOut{Input}{input}
\Input{$u_i^0$ for $i=1,...,n$}
\For{$k=0,1,...$ until convergence}{
\For{$i=1,...,n$}{
Find $u_i^{k+1}$ s.t.
\begin{alignat*}{2}
-\nlOp u_i^{k+1} &= f_i \quad && \text{on } \nlDom_i, \\
u_i^{k+1} &= u_j^{k+1} \quad && \text{on } \nlBound_i \cap \nlDom_j \text{ for }j=1,...,i-1, \\
u_i^{k+1} &= u_j^{k} \quad && \text{on } \nlBound_i \cap \nlDom_j \text{ for }j=i+1,...,n, \\
u_i^{k+1} &= g \quad && \text{on } \nlBound_i \cap \nlBound.
\end{alignat*}
}
}
\caption{Multiplicative Schwarz method}
\end{algorithm}~\\
Another approach is to only use the solutions of the former outer iteration $\{u_j^{k}\}_{j=1}^n$ instead of utilizing $u_j^{k+1}$ for the subproblem $i$, if $0 \leq j < i$. This leads to the so-called additive Schwarz algorithm that is presented in Algorithm \ref{alg:additive_schwarz}. This version of Schwarz iterative methods is easily parallelizable. In some cases there is also a way to compute the multiplicative Schwarz algorithm in parallel to some extend, which leads to the multicolor Schwarz algorithm that is illustrated in \cite[Algorithm 2.2.2]{mathewDD} for a coercive elliptic partial differential equation.\\
\begin{algorithm}[H]
\label{alg:additive_schwarz}
\SetAlgoLined
\DontPrintSemicolon
\SetKwInOut{Input}{input}
\Input{$u_i^0$ for $i=1,...,n$}
\For{$k=0,1,...$ until convergence}{
\For{$i=1,...,n$}{
Find $u_i^{k+1}$ s.t.
\begin{alignat*}{2}
-\nlOp u_i^{k+1} &= f_i \quad && \text{on } \nlDom_i, \\
u_i^{k+1} &= u_j^{k} \quad &&\text{on } \nlBound_i \cap \nlDom_j \text{ for }j=1,...,n, \\ 
u_i^{k+1} &= g \quad && \text{on } \nlBound_i \cap \nlBound.
\end{alignat*}
}
}
\caption{Additive Schwarz Method}
\end{algorithm}
\begin{remark}
\label{remark:add_schwarz_two_domains}
    If we decompose $\nlDom$ in only two domains $\nlDom_1$ and $\nlDom_2$ the additive Schwarz method basically consists of two multiplicative Schwarz methods, where one starts by solving the subproblem on $\nlDom_1$ and the second one begins by computing the solution on  $\nlDom_2$. 
    Therefore, the application of the additive Schwarz methods can only be faster, if the domain $\nlDom$ is decomposed in at least three subdomains.
\end{remark}~\\
In Chapter \ref{chap:schwarz_finite_elements} we will present the finite element versions of the multiplicative and additive Schwarz method. Then, the discretized multiplicative Schwarz method coincides with the block-Gauß-Seidel algorithm and the discretized additive Schwarz approach is equivalent to the block-Jacobi method.
\subsection{Schwarz Methods for Nonlocal Problems with Neumann Boundary Conditions}
\label{sec:Neumann}
In this section we apply the Schwarz Method on a nonlocal problem with Neumann boundary conditions as introduced in Chapter \ref{chap:nl_prob_robin} For ease of presentation, we assume $\nlBound^N = \nlBound$ (and consequently $\nlBound^D = \emptyset$). The case, where $\nlBound = \nlBound^N \dot{\cup} \nlBound^D$ with $\nlBound^D \neq \emptyset$, can be handled analogously. Thus, we consider the following problem.
\begin{align}
	\begin{split}
	\label{nlProb_NeumannBoundary}
	\textit{Find a weak solution } u &\in \trialSpace(\nlDom, \nlBound, \emptyset) \textit{ to}\\
	-\nlOp u + \kappa u &= f \textit{ on } \nlDom, \\
	\neumannOp u &= g^N \textit{ on } \nlBound,
	\end{split}
\end{align}
where $\kernel(\xb,\yb)=0$ for $\xb,\yb \in \nlBound$, i.e., we use the nonlocal Neumann operator \eqref{vu_neumann_op}.\\
Moreover, let $\{\nlBound_i^N\}_{i=1}^n$ be a partition of the nonlocal boundary $\nlBound(=\nlBound^N)$ such that
\begin{align*}
	&\overline{\nlBound} = \left( \bigcup_{i=1}^n \overline{\nlBound}_i^N \right),~ \nlBound_i^N \cap \nlBound_j^N = \emptyset \text{ for } i,j = 1,...,n, ~i \neq j\text{ and} \\
	&\nlBound_i^N \subset \nlBound_i \cap \nlBound = \{ \yb \in \nlBound: \int_{\nlDom_i}  \kernel(\xb, \yb) ~d\xb > 0 \}.
\end{align*}
Here, it is possible that there exist a subset of indices $\mathbb{S} \subset \{1,...,n\}$ with $|\mathbb{S}| \leq n-1$ such that 
\begin{align*}
\nlBound^N_{j} = \emptyset \textit{ for } j \in \mathbb{S}.
\end{align*}
Thus, the partition $\{\nlBound_i^N\}_{i=1}^n$ can contain the empty set multiple times and includes at least one non-empty set. Moreover, the boundary data $g^N$ may vary across the different parts of the decomposition $\{\nlBound_i^N\}_{i=1}^n$, i.e., there exist mutually different $g_i^N:\nlBound_i^N \rightarrow \mbR$, for $i=1,...,n,$ such that $g^N=\sum_{i=1}^n g_i^N \ind_{\nlBound_i^N}$.\\
Given the weak solution $u$ of the problem \eqref{nlProb_NeumannBoundary}, then the functions $u_i \defas u \ind_{\nlDom_i \cup \nlBound_i }$ for $i=1,...,n$ are solutions of
	\begin{align}
		\label{schwarz_formulation_neumann}
		\begin{split}
		\textit{Find a weak solution } u_i &\in \trialSpace(\nlDom_i, \nlBound_i^N, \nlBound_i \setminus \nlBound_i^N) \textit{~subject to}\\
		-\nlOp u_i + \kappa u_i &= f_i \textit{ on } \nlDom_i, \\
		\neumannOp u_i &= g_i^N \textit{ on } \nlBound_i^N, \\
		u_i &= u_j \textit{ on } \nlBound_i \cap \left(\nlDom_j \cup \nlBound_j^N \right), \textit{ for } j=1,...,n \textit{ and } j \neq i.
\end{split}	
\end{align}
So, on every subproblem we are solving a nonlocal Robin problem. Additionally, since the nonlocal boundary regarding $\nlDom_i$ can be partially in another subdomain $\nlDom_j$, i.e., $\nlBound_i \cap \nlDom_j \neq \emptyset$, the integration domain of the Neumann operator corresponding to the subproblem on $\nlDom_i$ can contain parts of the nonlocal Dirichlet Boundary and the nonlocal Neumann operator is in this case of type \eqref{neumann_op} and not of type \eqref{vu_neumann_op}.
Again, if we have a weak solutions $(u_i)_{i=1}^n$ of the problems \eqref{schwarz_formulation_neumann}, we set
\begin{align*}
	u = \sum_{i=1}^n u_i \ind_{\nlDom_i \cup \nlBound_i^N}
\end{align*}
and then $u$ is weak solution to \eqref{schwarz_formulation_neumann}, which we show in Section \ref{sec:well-posedness_neumann}. 
Again, analogously to the nonlocal Dirichlet case as described in Remark \ref{remark:boundary_extension}, we extend the Dirichlet boundary $\nlBound_i \setminus \nlBound_i^N$ corresponding to the subproblem on $\nlDom_i \cup \nlBound_i^N$ to the domain ${\extnlBound_i \defas \bigcup\limits_{\substack{j=1 \\ j \neq i}}^n\left(\nlDom_j \cup \nlBound_j^N\right) =\left(\nlDom \cup \nlBound \right) \setminus \left( \nlDom_i \cup \nlBound_i^N \right)}$. Then, $\trialSpace(\nlDom,\nlBound,\emptyset) \subset \trialSpace(\nlDom_i, \nlBound_i^N, \extnlBound_i)$ and, due to the symmetry of $\kernel$, we obtain for $\weakSol_i \in \trialSpace(\nlDom_i, \nlBound_i^N, \extnlBound_i)$ and $\advar_i \in \testSpace(\nlDom_i, \nlBound_i^N, \extnlBound_i)$ that   
\begin{align*}
	\neumannVarOp(\weakSol_i,\advar_i) + \int_{\nlDom} \kappa(\xb) \weakSol_i(\xb) \advar_i(\xb) ~d\xb = &\frac{1}{2} \int_{\nlDom_i \cup \nlBound_i^N} \int_{\nlDom_i \cup \nlBound_i^N} \left( \weakSol_i(\xb) - \weakSol_i(\yb) \right) \left( \advar_i(\xb) - \advar_i(\yb) \right) \kernel(\xb,\yb) ~d\yb d\xb \\
	&+ \sum_{\substack{j=1 \\ j \neq i}}^n \int_{\nlDom_i \cup \nlBound_i^N} \int_{\nlDom_j \cup \nlBound_j^N} \advar_i(\xb) \left(\weakSol_i(\xb) - \weakSol_i(\yb) \right) \kernel(\xb,\yb) ~d\yb d\xb \\
	&+ \int_{\nlDom_i} \kappa(\xb)\weakSol_i(\xb)\advar_i(\xb) ~d\xb.
\end{align*}
Thus, we define the variational formulation of the Robin problem on $\nlDom_i \cup \nlBound_i^N$ as:
\begin{align}
	\textit{Given } f \in L^2(\nlDom), g^N \in L^2(\nlBound^N) \textit{ and } g^D \in \trialSpace(\nlDom_i, \nlBound_i^N&, \hat{\nlBound}_i), \textit{ find } \weakSol_i \in \trialSpace(\nlDom_i, \nlBound_i^N, \hat{\nlBound}_i) \textit{ with } \nonumber \\
	\label{neumann_weak_formulation_subdomain}
	\begin{split}
	\neumannVarOp(\weakSol_i, \advar_i) + \int_{\nlDom_i} \kappa(\xb)\advar_i(\xb)\weakSol_i(\xb) ~d\xb &= \varForce^{N}(\advar_i) \textit{ for all } \advar_i \in \testSpace(\nlDom_i, \nlBound_i^N, \hat{\nlBound}_i)\\
	\textit{and } \weakSol_i - g^D &\in \testSpace(\nlDom_i, \nlBound_i^N, \hat{\nlBound}_i).
\end{split}
\end{align}
The multiplicative Schwarz algorithm can then be formulated as follows:\\
\begin{algorithm}[H]
	\label{alg:multiplicative_schwarz_neumann}
	\SetAlgoLined
	\DontPrintSemicolon
	\SetKwInOut{Input}{input}
	\Input{$u_i^0$ for $i=1,...,n$}
	\For{$k=0,1,...$ until convergence}{
		\For{$i=1,...,n$}{
			Find weak solution $u_i^{k+1}$ to
			\begin{alignat*}{2}
				-\nlOp u_i^{k+1} + \kappa u_i^{k+1} &= f_i \quad && \text{on } \nlDom_i, \\
				\neumannOp u_i^{k+1} &= g_i^N && \text{on } \nlBound_i^N,\\
				u_i^{k+1} &= u_j^{k+1} \quad && \text{on } \nlBound_i \cap \nlDom_j \text{ for }j=1,...,i-1, \\
				u_i^{k+1} &= u_j^{k} \quad && \text{on } \nlBound_i \cap \nlDom_j \text{ for }j=i+1,...,n.
			\end{alignat*}
		}
	}
	\caption{Multiplicative Schwarz Method for Nonlocal Problems with Neumann Boundary Conditions}
\end{algorithm}~\\
The additive Schwarz method can be analogously derived as in Section \ref{chap:schwarz_dirichlet} by only using the solution of the former outer iteration $(u_i^k)_{i=1}^n$ as the boundary data in iteration $k$ to compute the new solutions $(u_i^{k+1})_{i=1}^n$.

\section{Well-posedness of the Multiplicative Schwarz Method}
\label{sec:well_posedness}
In this section we show that the multiplicative Schwarz methods as described in Chapters \ref{chap:schwarz_dirichlet} and \ref{sec:Neumann} converge. Therefore we recall the first part of \cite[Theorem I.2]{LionsSchwarz1}:
\begin{theorem}
\label{theo:LionsTheorem}
Given closed subspaces $\{\schwarzSubspace_i\}_{i=1}^n$ of a Hilbert space $\schwarzSubspace \defas \overline{\oplus_{i=1}^n \schwarzSubspace_i}$, orthogonal projections $\projection_i:\schwarzSubspace \rightarrow \schwarzSubspace$ onto the subspace $\schwarzSubspace_i$ and a sequence $\left(\weakSol^l\right)_{l=1}^{\infty}$, where $\projection_i(\weakSol^{l-1})=\weakSol^l$ for $l=km+i$ given some data $\weakSol^0$. Then $\weakSol^l$ converges to a unique $\weakSol \in \schwarzSubspace$.
\end{theorem}~\\
The version of Theorem \ref{theo:LionsTheorem} for two domains (see \cite[Theorem I.1]{LionsSchwarz1}) has also been used to show the convergence of the multiplicative Schwarz method for an energy-based Local-to-Nonlocal coupling (see \cite{AcostaDD}).
\begin{remark}
    The second part of \cite[Theorem I.2]{LionsSchwarz1} states that if $\schwarzSubspace = \oplus_{i=1}^n \schwarzSubspace_i$ holds, then there exists a constant $\epsilon \in (0,1)$ such that
    \begin{align*}
        ||\weakSol^l - \weakSol||_{\schwarzSubspace} \leq \epsilon^l ||\weakSol^0 - \weakSol||_{\schwarzSubspace}.
    \end{align*}
    Linear convergence of the Schwarz methods in the nonlocal framework of Chapter \ref{chap:schwarz_methods} can also be observed in the numerical experiments in Chapter \ref{chap:numerical_experiments} and is proven in the discrete case(see Remark \ref{remark:conv_discrete_schwarz}). However, the proof of the assumption $\schwarzSubspace = \oplus_{i=1}^n \schwarzSubspace_i$ is left to future work.
\end{remark}
\subsection{Well-posedness of the Multiplicative Schwarz Method for Nonlocal Dirichlet Problems}
\label{well-posedness_dirichlet}
Here in this subsection, we assume the kernel $\kernel$ to be as described in Remark \ref{remark:examples_well_posedness}, i.e., $\kernel$ fulfills conditions (K1) and (K2) and is either a singular symmetric or an integrable symmetric kernel.
As seen in \eqref{weak_formulation_homogeneous} we only need to consider the case $g = 0$ on $\nlBound$ since the case $g \neq 0$ on $\nlBound$ is a direct consequence.
Given Algorithm \ref{alg:multiplicative_schwarz} we can define for every subproblem in \eqref{schwarz_formulation} an operator
\begin{align*} 
    \schwarzOp^f_i: \left(\testSpace(\nlDom \cup \nlBound), |\cdot|_{\trialSpace(\nlDom \cup \nlBound)} \right) \rightarrow \left(\testSpace(\nlDom \cup \nlBound), |\cdot|_{\trialSpace(\nlDom \cup \nlBound)} \right),
\end{align*} such that $\schwarzOp^f_i(g) = \weakSol_i$ and $\weakSol_i:\nlDom_i \cup \extnlBound_i \rightarrow \mbR$ is the weak solution of the subproblem on $\nlDom_i$  as described in \eqref{weak_formulation_subdomain} given $g$ and $f$, i.e., $\schwarzOp_i$ maps the boundary data $g$ onto the solution $\weakSol_i$ on $\nlDom_i$.\\
We start by showing the convergence for $f \equiv 0$.
In this case, we set $\schwarzOp_i \defas \schwarzOp^f_i$.
\begin{lemma}
The operator $\schwarzOp_i$ is linear and bounded as follows
\begin{align*}
	||\weakSol_i||_{\trialSpace(\completeDom)} \leq  2| g|_{\trialSpace(\completeDom)} \text{ for } \weakSol_i = \schwarzOp_i(g).
\end{align*}	
\end{lemma}
\begin{proof}
    Since the linearity of $\schwarzOp_i$ is trivially true, we only have to proof the boundedness of $\schwarzOp_i$. As a consequence of $f \equiv 0$ and \eqref{weak_formulation_homogeneous} in the second step of the following calculation and by applying Cauchy-Schwarz in the third step  we derive
	\begin{align*}
		&|\weakSol_i - g|_{\trialSpace(\completeDom)}^2 = \varOp(\weakSol_i - g,\weakSol_i - g) \leq |A(g, \weakSol_i-g)| \\
		&\leq \sqrt{A(g,g)}\sqrt{A(\weakSol_i-g,\weakSol_i-g)} = |g|_{\trialSpace(\completeDom)} |\weakSol_i-g|_{\trialSpace(\completeDom)} \text{ and therefore } |\weakSol_i-g|_{\trialSpace(\completeDom)} \leq |g|_{\trialSpace(\completeDom)}.
	\end{align*}
	Then, we conclude
	\begin{align*}
		|\weakSol_i|_{\trialSpace(\completeDom)} \leq |\weakSol_i - g|_{\trialSpace(\completeDom)} + |g|_{\trialSpace(\completeDom)} \leq 2 |g|_{\trialSpace(\completeDom)}. 
	\end{align*}
\end{proof}~\\
Further, we define spaces
\begin{align}
    \schwarzSubspace_i(\completeDom) \defas \{\weakSol_i \in \testSpace(\completeDom): \weakSol_i = \schwarzOp_i(\weakSol_i)  \} ~~\text{ and }~~ \schwarzSubspace(\completeDom) \defas \overline{\oplus_{i=1}^n \schwarzSubspace_i(\completeDom)}, \label{def:schwarzSubspace_dirichlet}
\end{align}
that is equipped with the the inner product
\begin{align*}
    &\langle \weakSol,\advar \rangle_{\schwarzSubspace(\completeDom)} \defas \varOp\left( \weakSol, v \right) \text{ and the norm } ||\weakSol||_{\schwarzSubspace(\completeDom)} \defas \sqrt{\langle \weakSol, \weakSol \rangle}_{\schwarzSubspace(\completeDom)}.
\end{align*}
Then $\langle \cdot, \cdot \rangle_{\schwarzSubspace(\completeDom)}$ is linear in the first argument, symmetric and positive-definite, see Remark \ref{remark:examples_well_posedness}.
Here, every $\weakSol_i \in \schwarzSubspace_i(\completeDom)$ is a solution to \eqref{weak_formulation_subdomain} on $\nlDom_i$ given $\weakSol_i \in \trialSpace(\nlDom_i \cup \extnlBound_i)$ as boundary data on $\extnlBound_i$ and $f \equiv 0$.
\begin{lemma}
	\label{lemma:subspaces_dirichlet}
 Let the spaces $\schwarzSubspace_i(\completeDom)$ and $\schwarzSubspace(\completeDom)$ be as defined in \eqref{def:schwarzSubspace_dirichlet} for $i=1,...,n$. Then, the space
 $\schwarzSubspace_i(\completeDom)$ is a closed subspace for every $i=1,...,n$. Moreover, $\schwarzSubspace(\completeDom)$ is a closed subspace of $\testSpace(\completeDom)$ and therefore a Hilbert space.   
\end{lemma}
\begin{proof}
    \begin{enumerate}
        \item $\schwarzSubspace_i(\completeDom)$ is closed under linear combinations: Given $\weakSol_i, \psi_i \in \schwarzSubspace_i(\completeDom)$ and $\lambda \in \mbR$, then $\lambda \weakSol_i + \psi_i \in \testSpace(\completeDom)$ and we get 
        \begin{align*}
            \varOp (\lambda \weakSol_i + \psi_i, \advar_i ) &= \lambda \varOp( \weakSol_i, \advar_i) + \varOp(\psi_i, \advar_i) 
            = 0 \text{ for all } \advar_i \in \testSpace(\nlDom_i \cup \hat{\nlBound}_i)\\
            &\Rightarrow \lambda \weakSol_i + \psi_i \in \schwarzSubspace_i(\completeDom).
        \end{align*}
        \item $\schwarzSubspace_i(\completeDom)$ is a closed subspace regarding $||\cdot||_{\schwarzSubspace(\completeDom)}$: Given a sequence $\{\weakSol_i^m\}_{m=1}^\infty$ such that $\weakSol_i^m \in \schwarzSubspace_i(\completeDom)$ for all $m \in \Nbb$ and  $\{\weakSol_i^m\}_{m=1}^\infty$ is a Cauchy sequence regarding $||\cdot||_{\schwarzSubspace(\completeDom)}$.
        Then, there exists a unique function $\weakSol_i \in \schwarzSubspace(\completeDom)$ with $\lim\limits_{m \rightarrow \infty}|\weakSol_i^m - \weakSol_i|_{\trialSpace(\completeDom)}=0$, because $\left( \weakSol_i^m \right)_{m=1}^\infty$ is a Cauchy sequence in the Hilbert space $\left(\testSpace(\completeDom), |\cdot|_{\trialSpace(\completeDom)}\right)$. Consequently, the continuity of $\schwarzOp_i$ implies $\weakSol_i \in \schwarzSubspace_i(\completeDom)$ since
        \begin{align*}
            |\weakSol_i - \schwarzOp_i(\weakSol_i)|_{\testSpace(\completeDom)} &\leq |\weakSol_i - \weakSol_i^m|_{\trialSpace(\completeDom)} + |\schwarzOp_i(\weakSol_i^m) - \schwarzOp_i(\weakSol_i)|_{\trialSpace(\completeDom)}\\
            &\leq 3 |\weakSol_i - \weakSol_i^m|_{\trialSpace(\completeDom)} \rightarrow 0 \quad \text{for } m \rightarrow \infty. 
        \end{align*}
        \item $\schwarzSubspace(\completeDom)$ is closed regarding linear combinations, which follows directly from the definition of $\schwarzSubspace(\completeDom)$ and the fact that every $\schwarzSubspace_i(\completeDom)$ is closed under linear combinations. Then, $\schwarzSubspace(\completeDom)$ is a Hilbert space since it is closed regarding $||\cdot ||_{\schwarzSubspace(\completeDom)}$.
    \end{enumerate}
\end{proof}~\\
We can now conclude the following:
\begin{lemma}
For $i=1,...,n$ the function $\schwarzOp_i$ is an orthogonal projection onto $\schwarzSubspace_i(\completeDom)$ regarding the inner product $\langle \cdot,\cdot \rangle_{\schwarzSubspace(\completeDom)}$.
\end{lemma}
\begin{proof}
Obviously $\schwarzOp_i(g) \in \schwarzSubspace_i(\completeDom)$ for all $g \in \schwarzSubspace(\completeDom)$. 
In order to prove the orthogonality condition $g - \schwarzOp_i(g) = \left(g -  \schwarzOp_i(g)\right)\ind_{\Omega_i} \perp \schwarzSubspace_i(\completeDom)$ we set $v_i \defas g -  \schwarzOp_i(g) \in \testSpace(\nlDom_i \cup \hat{\nlBound}_i)$, then for all $\weakSol_i \in \schwarzSubspace_i(\completeDom)$ we get
\begin{align*}
    \langle \weakSol_i, g - \schwarzOp_i(g) \rangle_{\schwarzSubspace(\completeDom)} = \varOp(\weakSol_i, v_i) = \varForce(v_i) = 0,
\end{align*}
since $v_i \in \testSpace(\nlDom_i \cup \hat{\nlBound}_i)$ can be interpreted as a test function and $\weakSol_i$ with $\weakSol_i = \schwarzOp_i(\weakSol_i)$ solves the subproblem on $\nlDom_i$ as defined in \eqref{weak_formulation_subdomain} given $\weakSol_i$ as boundary data and $f_i=0$.
\end{proof}
\begin{corollary}
	In the case of integrable symmetric or singular symmetric kernels for the solution $\weakSol \in \testSpace(\completeDom)$ of \eqref{weak_formulation_homogeneous} trivially also holds $\weakSol \in \schwarzSubspace(\completeDom)$. Therefore, the multiplicative Schwarz method yields a sequence $\left(\weakSol^l\right)_{l=1}^\infty$, where $\weakSol^l \in \schwarzSubspace(\completeDom)$ for $l \in \Nbb$, that converges to $\weakSol$ regarding $||\cdot||_{\schwarzSubspace(\completeDom)}=|\cdot|_{\trialSpace(\completeDom)}$.
\end{corollary}
\begin{remark}
	\label{remark:conv_general_f}
	For $f \nequiv 0$ the multiplicative Schwarz method yields a sequence $\{\weakSol^l\}_{l=1}^\infty$ such that for $l=km + i$ the function $\weakSol^l$ fulfills $\weakSol^l = \schwarzOp_i^f(\weakSol^{l-1})$ and $\weakSol^l - \weakSol^{l-1} \in \testSpace(\nlDom_i \cup \extnlBound_i)$, i.e., $\weakSol^l$ solves \eqref{weak_formulation_subdomain} on $\nlDom_i$ given $f$ and boundary data $\weakSol^{l-1}$. Let $\weakSol$ be the solution of the problem on the whole domain. Then for $l=km+i$ the function $\tilde{\weakSol}^l \defas \weakSol^l - \weakSol$ fulfills
	\begin{align*}
		\varOp(\tilde{\weakSol}^l, \advar_i) = \varOp(\weakSol^l - \weakSol, \advar_i) = \varForce(\advar_i) - \varForce(\advar_i) = 0 \text{ for all } \advar_i \in \testSpace(\nlDom_i \cup \hat{\nlBound}_i).
	\end{align*}
Therefore $\{\tilde{\weakSol}^l\}_{l=1}^\infty$ is a sequence generated by the multiplicative Schwarz method which converges to zero, i.e., $\lim\limits_{l \rightarrow \infty} \weakSol^l - \weakSol = 0$.
\end{remark}
\subsection{Well-posedness of the Multiplicative Schwarz Method for Nonlocal Problems with Neumann Boundary Conditions}
\label{sec:well-posedness_neumann}
As mentioned in Chapters \ref{chap:nl_prob_robin} and \ref{sec:Neumann} we only consider the kernel $\kernel$ to be symmetric in nonlocal problems of type \eqref{nlProb_NeumannBoundary}.
Next, we define 
\begin{align*}
 \schwarzOp_i^{N,f,g^N} : \left(\trialSpace(\nlDom_i, \nlBound_i^N, \extnlBound_i), ||\cdot||_{\trialSpace(\nlDom_i, \nlBound_i^N, \extnlBound_i)} \right) \rightarrow \left(\trialSpace(\nlDom_i, \nlBound_i^N, \extnlBound_i), ||\cdot||_{\trialSpace(\nlDom_i, \nlBound_i^N, \extnlBound_i)} \right)
\end{align*}
such that $\schwarzOp_i^{N,f,g^N}\left( g \right) = \weakSol_i$ and $\weakSol_i : \completeDom \rightarrow \mbR$ is the weak solution of the subproblem regarding $\nlDom_i$ as described in \eqref{neumann_weak_formulation_subdomain}. 
We will only show the convergence for $f,g^N \equiv 0$, since the convergence for $f \in L^2(\nlDom)$ and $g^N \in L^2(\nlBound)$ with $f,g^N \notequiv 0$ is a direct conclusion, which can be shown analogously to Remark \ref{remark:conv_general_f}. In this case, we set $\schwarzOp_i^N \defas \schwarzOp_i^{N,f,g^N}$.
\begin{lemma}
	The operator $\schwarzOp_i^N$ is linear and bounded by the boundary data, i.e., there exists a constant $C > 0$ with 
	\begin{align}
		\label{neumann_boundedness}
		||\weakSol_i||_{\trialSpace(\nlDom_i, \nlBound_i^N, \extnlBound_i)} \leq C || g ||_{\trialSpace(\nlDom_i, \nlBound_i^N, \extnlBound_i)}, \text{ if } \weakSol_i = \schwarzOp_i^N(g) \text{ and } g \in \trialSpace(\nlDom_i,\nlBound_i^N,\extnlBound_i).
	\end{align}
\end{lemma}
\begin{proof}
	Set $\widehat{C} \defas \max\{2, \frac{1}{\alpha}\}$ and notice that $\weakSol_i - g \in \testSpace(\nlDom_i, \nlBound_i^N, \hat{\nlBound}_i)$. As a consequence, we get by applying the norm equivalence \eqref{ineq:neumann_norms} and formulation \eqref{neumannVarFormulation_homogeneous} on the space $\trialSpace(\nlDom_i,\nlBound_i^N,\extnlBound_i)$ that
	\begin{align*}
		|| \weakSol_i - g||^2_{\trialSpace(\nlDom_i,\nlBound_i^N, \extnlBound_i)} &\leq \widehat{C} \left( \neumannVarOp(\weakSol_i - g, \weakSol_i - g) + \int_{\nlDom_i} \kappa \left(\weakSol_i - g \right)^2 ~d\xb \right) \\
		 &\leq \widehat{C} \left| \neumannVarOp(g, \weakSol_i - g) \right| + \widehat{C} \left| \int_{\nlDom_i} \kappa g \left( \weakSol_i - g \right) ~d\xb \right| \\
		&\leq \widehat{C}\sqrt{\neumannVarOp(g,g)} \sqrt{\neumannVarOp(\weakSol_i - g,\weakSol_i - g)} + \widehat{C}\sqrt{\int_{\nlDom_i} \kappa g^2 ~d\xb}\sqrt{\int_{\nlDom_i} \kappa \left( \weakSol_i - g \right)^2 ~d\xb} \\
		&\leq \widehat{C}||g||_{\trialSpace(\nlDom_i, \nlBound_i^N, \extnlBound_i)} ||\weakSol_i - g||_{\trialSpace(\nlDom_i,\nlBound_i^N,\extnlBound_i)},
	\end{align*}
where we used Cauchy-Schwarz in the third step. Then, the assertion follows from
\begin{align*}
	||\weakSol_i||_{\trialSpace(\nlDom_i, \nlBound_i^N, \extnlBound_i)} \leq || \weakSol_i - g||_{\trialSpace(\nlDom_i, \nlBound_i^N, \extnlBound_i)} + ||g||_{\trialSpace(\nlDom_i, \nlBound_i^N, \extnlBound_i)}.
\end{align*}
\end{proof}~\\
Additionally define
\begin{align}
	\schwarzSubspace_i^N(\nlDom,\nlBound,\emptyset) \defas \{ \weakSol_i \in \trialSpace(\nlDom,\nlBound,\emptyset) : \weakSol_i = \schwarzOp_i^N\left( \weakSol_i \right) \} \text{ and } \schwarzSubspace^N(\nlDom,\nlBound,\emptyset) \defas \overline{\oplus_{i=1}^n \schwarzSubspace_i^N(\nlDom,\nlBound,\emptyset)}, \label{def:schwarzSubspace_neumann}
\end{align}
which are equipped with the inner product $\langle \weakSol, \advar \rangle_{\schwarzSubspace^N(\nlDom,\nlBound,\emptyset)} \defas \neumannVarOp(\weakSol, \advar) + \int_{\nlDom} \kappa \weakSol \advar ~d\xb$ and the norm  $||\weakSol||_{\schwarzSubspace^N(\nlDom,\nlBound,\emptyset)} \defas \sqrt{\langle \weakSol, \weakSol \rangle}_{\schwarzSubspace^N(\nlDom,\nlBound,\emptyset)} \left(= |||\weakSol|||_{\trialSpace(\nlDom, \nlBound, \emptyset)} \right)$. Here, every function $\weakSol_i \in \schwarzSubspace_i^N(\nlDom,\nlBound,\emptyset)$ is a solution to \eqref{neumann_weak_formulation_subdomain} on $\nlDom_i \cup \nlBound_i^N$ given $\weakSol_i \in \trialSpace(\nlDom_i, \nlBound_i^N, \hat{\nlBound}_i)$ as boundary data on $\extnlBound_i$ and $f \equiv 0$.
\begin{lemma}
Let the spaces $\schwarzSubspace_i^N(\nlDom,\nlBound,\emptyset)$ and $\schwarzSubspace^N(\nlDom,\nlBound,\emptyset)$ be as defined in \eqref{def:schwarzSubspace_neumann} for every ${i=1,...,n}$.
Then, the space $\schwarzSubspace_i^N(\nlDom,\nlBound,\emptyset)$ is a closed subspace for $i=1,...,n$ and $\schwarzSubspace^N(\nlDom,\nlBound,\emptyset)$ is a closed subspace of $\trialSpace(\nlDom,\nlBound,\emptyset)$ and therefore a Hilbert space.
\begin{proof}
	\begin{enumerate}
		\item $\schwarzSubspace_i^N(\nlDom,\nlBound,\emptyset)$ is closed under linear combinations. Given $\weakSol, \psi \in \schwarzSubspace_i^N(\nlDom,\nlBound,\emptyset)$ and $\lambda \in \mbR$, then $\lambda \weakSol + \psi \in \trialSpace(\nlDom,\nlBound,\emptyset)$ and we get 
		\begin{align*}
			&\neumannVarOp(\lambda \weakSol + \psi, \advar_i )+ \int_{\nlDom}\kappa \advar_i \left( \lambda \weakSol + \psi \right) ~d\xb \\
			&= \lambda \left( \neumannVarOp( \weakSol, \advar_i)+  \int_{\nlDom}\kappa \advar_i \weakSol ~d\xb \right)\\
			&+ \neumannVarOp(\psi, \advar_i) + \int_{\nlDom} \kappa \advar_i \psi ~d\xb
			= 0 \text{ for all } \advar_i \in \testSpace(\nlDom_i, \nlBound_i^N,\extnlBound_i)\\
			&\Rightarrow \lambda \weakSol + \psi \in \schwarzSubspace_i^N.
		\end{align*}
		\item $\schwarzSubspace_i^N(\nlDom,\nlBound,\emptyset)$ is closed with regards to $||\cdot||_{\schwarzSubspace^N(\nlDom,\nlBound,\emptyset)}$. Given a Cauchy sequence $\{\weakSol_i^m\}_{m=1}^\infty$ regarding $||\cdot ||_{\schwarzSubspace^N(\nlDom,\nlBound,\emptyset)}$. Due to the equivalence of $||\cdot||_{\schwarzSubspace^N(\nlDom,\nlBound,\emptyset)}=|||\cdot|||_{\trialSpace(\nlDom,\nlBound,\emptyset)}$ and $||\cdot||_{\trialSpace(\nlDom,\nlBound,\emptyset)}$, there exists a function $\weakSol_i \in \trialSpace(\nlDom, \nlBound, \emptyset)$ such that
		\begin{align*}
			\lim\limits_{m \rightarrow \infty}|| \weakSol_i^m - \weakSol_i||_{\trialSpace(\nlDom, \nlBound, \emptyset)} = 0.
		\end{align*}
	Then, by using the continuity of $\schwarzOp_i^N$ we get $\weakSol_i \in \schwarzSubspace_i^N(\nlDom,\nlBound,\emptyset)$ since
 \begin{align*}
     ||\weakSol_i - \schwarzOp^N_i(\weakSol_i)||_{\trialSpace(\nlDom, \nlBound, \emptyset)} &= ||\weakSol_i - \schwarzOp_i^N(\weakSol_i)||_{\trialSpace(\nlDom_i, \nlBound^N_i, \extnlBound_i)} \\
     &\leq ||\weakSol_i - \weakSol_i^m||_{\trialSpace(\nlDom_i, \nlBound^N_i, \extnlBound_i)} + ||\schwarzOp^N(\weakSol_i^m) - \schwarzOp_i^N(\weakSol_i)||_{\trialSpace(\nlDom_i, \nlBound^N_i, \extnlBound_i)} \\ 
     &\leq \left( 1 + C \right) ||\weakSol_i - \weakSol_i^m||_{\trialSpace(\nlDom_i, \nlBound^N_i, \extnlBound_i)} \rightarrow 0 \quad m \rightarrow \infty,
 \end{align*}
 where we used that $\weakSol_i = \schwarzOp_i^N(\weakSol_i)$ on $(\completeDom) \setminus \nlDom_i$ in the first step as well as \eqref{neumann_boundedness} and the fact that ${||\cdot||_{\trialSpace(\nlDom_i,\nlBound_i^N, \extnlBound_i)} \leq ||\cdot||_{\trialSpace(\nlDom,\nlBound,\emptyset)}}$ in the last two step.
	\item Analogously to Lemma \ref{lemma:subspaces_dirichlet}.
	\end{enumerate}
\end{proof}
\end{lemma}~\\
Lastly, we can deduce the following:
\begin{lemma}
	The function $\schwarzOp_i^N$ is an orthogonal projection onto $\schwarzSubspace_i^N(\nlDom,\nlBound,\emptyset)$ regarding the inner product $\langle \cdot, \cdot \rangle_{\schwarzSubspace^N(\nlDom,\nlBound,\emptyset)}$ for $i=1,...,n$.
\end{lemma}
\begin{proof}
	Obviously, $\schwarzOp_i^N(g) \in \schwarzSubspace_i^N(\nlDom,\nlBound,\emptyset)$ for all $g \in \schwarzSubspace^N(\nlDom,\nlBound,\emptyset)$.\\ 
 Set $\advar_i \defas g - \schwarzOp_i^N(g) = \left(g - \schwarzOp_i^N(g) \right)\ind_{\nlDom_i \cup \nlBound_i^N} \in \testSpace(\nlDom_i, \nlBound_i^N, \hat{\nlBound}_i)$. We now have to proof 
	\begin{align*}
		g - \schwarzOp_i^N(g)= \advar_i \perp \schwarzSubspace_i^N(\nlDom,\nlBound,\emptyset).
	\end{align*}
For all $\weakSol_i \in \schwarzSubspace_i^N(\nlDom,\nlBound,\emptyset)$ we get
\begin{align*}
	\langle \weakSol_i, g - \schwarzOp_i^N(g) \rangle_{\schwarzSubspace^N(\nlDom,\nlBound,\emptyset)}
	=\neumannVarOp(\weakSol_i,\advar_i) + \int_{\nlDom_i} \kappa\weakSol_i \advar_i = \varForce^N(\advar_i) = 0,
\end{align*}
since $\weakSol_i$ with $\weakSol_i = \schwarzOp_i^N(\weakSol_i)$ solves the variational formulation \eqref{neumann_weak_formulation_subdomain} on $\nlDom_i \cup \nlBound_i^N$ given the Dirichlet boundary data $\weakSol_i$ and $\advar_i \in \testSpace(\nlDom_i,\nlBound_i^N,\hat{\nlBound}_i)$ serves as a test function.
\end{proof}
\begin{corollary}
	The weak solution $\weakSol \in \trialSpace(\nlDom, \nlBound, \emptyset)$ of \eqref{nlProb_NeumannBoundary} is also an element of the space $\schwarzSubspace^N(\nlDom,\nlBound,\emptyset)$. As a consequence, the multiplicative Schwarz method yields a sequence $\left( \weakSol^l \right)_{l=1}^\infty$, where $\weakSol^l \in \schwarzSubspace^N(\nlDom,\nlBound,\emptyset)$ for $l \in \Nbb$, which converges to $\weakSol$ regarding
 ${||\cdot||_{\schwarzSubspace^N(\nlDom,\nlBound,\emptyset)}=|||\cdot|||_{\trialSpace(\nlDom, \nlBound, \emptyset)}}$ (or equivalently, $||\cdot||_{\trialSpace(\nlDom, \nlBound, \emptyset)}$).
\end{corollary}
\section{Schwarz Methods for the Finite Element Formulation} \label{chap:finite_element_method}
\subsection{Finite Element Approximation}
\label{subsec:FE_Approx}
In order to solve problem \eqref{weak_formulation} we use the finite element method. Therefore we need a triangulation of finite (polyhedral) elements $\Tc^h = \{\Ec_j \}_{j=1}^J$ with vertices $\{\xb_k\}_{k=1}^K$. We assume that $\nlDom = \bigcup_{j=1}^J \Ec_j$ or $\nlDom \approx \bigcup_{j=1}^J \Ec_j$ and, for ease of exposition, that there exists a $K_\nlDom \in \Nbb$ such that $\xb_k \in \nlDom$ for $k=1,...,K_\nlDom$ and $\xb_k \in \nlBound$ for $k=K_\nlDom+1,...,K$.
Furthermore, we utilize continuous piecewise linear basis functions $\{\basisfun_k\}_{k=1}^K$ that satisfy $\basisfun_k(\xb_l)=\delta_{kl}$. In our tests we use a slightly different set of basis functions, which we describe in Section \ref{chap:splitting}. Moreover, we define finite dimensional function spaces as
\begin{align*}
    \trialSpace^h = span\{\basisfun_k: k=1,...,K\} \text{ and } \testSpace^h = span\{\basisfun_k: k=1,...,K_\nlDom\}.
\end{align*}
Thus, we can project $u \in \trialSpace(\completeDom)$ to a function $\FEu^h \in \trialSpace^h$ as follows
\begin{alignat*}{2}
   \FEu_i^h \defas u(\xb_i),\ \text{for all } i=1,...,K, &\quad \text{and}\quad & \FEu^h(\xb) = \sum_{i=1}^K \FEu_i^h \basisfun_i(\xb).
\end{alignat*}
With the projection $\FEDirData^h= \sum_{j=K_\nlDom+1}^K \FEDirData_i \basisfun_i$ of the boundary data $g$, we formulate a discretized version of the weak formulation \eqref{weak_formulation_homogeneous}:\\
Find $\FEu_\nlDom^h \in \testSpace^h$ such that
\begin{align}
    \varOp(\FEu_\nlDom^h, \basisfun_i) &= \varForce(\basisfun_i) & \text{for all } i=1,...,K_\nlDom, \nonumber\\
    \Leftrightarrow \sum_{k=1}^{K_\nlDom} \varOp(\basisfun_k,\basisfun_i)\FEu_k^h &= \FEForceEntry_i^h - \sum_{j=K_{\nlDom} + 1}^K \varOp (\basisfun_j, \basisfun_i)\FEDirData_j^h, & \text{for all } i=1,...,K_\nlDom, \label{FE_weak_formulation}
\end{align}
where $\FEForceEntry_i^h \defas \int_\nlDom f(\xb) \basisfun_i(\xb) ~d\xb$. Then, $\FEu^h = \sum_{i=1}^{K_\nlDom} \FEu_i^h\basisfun_i + \sum_{j= K_\nlDom + 1}^K \FEDirData_j^h\basisfun_j$ is called finite element approximation of \eqref{weak_formulation_homogeneous}.\\
By defining a matrices $\FEMatrix,\FEMatrix_{\nlBound}$ and a forcing vector $\FEForce$ as
\begin{align*}
    \FEForce = \left( \FEForceEntry_i^h \right)_{1 \leq i \leq K_{\nlDom}}, \FEMatrix = \left( \FEMatrixEntry_{ij} \right)_{1 \leq i,j \leq K_{\nlDom}}, \text{ and }
    \FEMatrix_{\nlBound} = \left( \FEMatrixEntry_{ij} \right)_{1 \leq i \leq K_{\nlDom}; K_{\nlDom} + 1 \leq j \leq K},
\end{align*}
where $\FEMatrixEntry_{ij} \defas \varOp(\basisfun_j, \basisfun_i)$, we can write \eqref{FE_weak_formulation} as a system of linear equations
\begin{align}\label{linear_system}
    \FEMatrix \FEu^h = \FEForce - \FEMatrix_I \FEDirData^h.
\end{align}
A detailed description on how to assemble the stiffness matrices $\FEMatrix$ and $\FEMatrix_{\nlBound}$ can be found in \cite{vollmann_cookbook}.
\begin{remark}
	The finite element formulation for the nonlocal problem with Neumann boundary conditions \eqref{neumannVarFormulation_homogeneous} can be derived in a similar fashion by replacing the bilinear operator $\varOp(\weakSol,\advar)$ with the function $\tilde{\neumannVarOp}(\weakSol, \advar) \defas \neumannVarOp(\weakSol,\advar) + \int_{\nlDom} \kappa\weakSol\advar~d\xb$. Therefore, the Schwarz algorithms for the finite element method, as described in Section \ref{chap:schwarz_finite_elements}, can be employed analogously in this case.
\end{remark}
\begin{remark}
    \label{remark:pos_def_stiff_matrix}
    In the case of a singular symmetric or an integrable symmetric kernel as introduced in Remark \ref{remark:examples_well_posedness} the corresponding variational operator $\varOp$ is symmetric and positive definite and, as a result, $\FEMatrix$ is also symmetric and positive definite. If we instead consider the nonlocal problem with Neumann boundary condition, where the kernel $\kernel$ is only assumed to be symmetric, we also derive that the corresponding bilinear operator $\tilde{\neumannVarOp}(\weakSol,\advar) = \neumannVarOp(\weakSol,\advar) + \int_{\nlDom} \kappa\weakSol\advar~d\xb$ and therefore the resulting finite element matrix is positive definite and symmetric (see Chapter \ref{chap:nl_prob_robin}). 
\end{remark}
\subsection{Splitting of Inner Boundary Basis Functions}\label{chap:splitting}
For ease of exposition, we now assume $\bar{\nlDom} = \bar{\nlDom}_1 \cup \bar{\nlDom}_2$ in this section.
Since we use the continuous Galerkin method in order to solve the variational formulation, the support of a piecewise linear nodal function $\basisfun_k$ corresponding to a vertex $\xb_k$, that lies on the boundary between $\nlDom_1$ and $\nlDom_2$, intersects with both domains, i.e., $supp(\basisfun_k) \cap \nlDom_i \neq \emptyset$, for $i=1,2$.
Therefore, we replace every one of these basis function $\basisfun_k$, whose support intersects with both domains, by two functions $\basisfun_k^1$ and $\basisfun_k^2$, that fulfill
\begin{align*}
	\basisfun_k(\xb) = \basisfun_k^1(\xb) + \basisfun^2_k(\xb) \text{ and } supp(\basisfun_k^i) \subset \bar{\nlDom}_i, \text{ for } i=1,2.
\end{align*}
Additionally, we replace $\FEu_k^h$ by two new degrees of freedom $\FEu_k^{h,i}$, for $i=1,2$. If we now follow the derivation of the finite element formulation analogously to Section \ref{subsec:FE_Approx} we get again a system of linear equations, which contains a linear equation for each $\FEu_k^{h,i}$.

\subsection{Formulation of the Schwarz Methods}
\label{chap:schwarz_finite_elements}
In this section we assume that the basis functions and the corresponding degrees of freedom are constructed as described in the previous Section \ref{chap:splitting}.
Then, in order to formulate the additive and multiplicative Schwarz methods for the linear system \eqref{linear_system} we now assign every node $\xb_k$ to a domain $\Omega_i$ according to the following index sets
\begin{align*}
    \indexSet_i &= \{j_1,...,j_{K_i}\} \subset \{1,...,K_\Omega\}, \text{ for } i=1,...,n, \text{ with}\\
    \bigcup_{i=1}^n \indexSet_i &= \{1,...,K_{\Omega}\} \text{ and } \indexSet_i \cap \indexSet_j = \emptyset, \text{ for } i,j=1,...,n \text{ and } i \neq j,
\end{align*}
where additionally the implication
    $supp (\basisfun_k) \subset \bar{\nlDom}_i \Rightarrow k \in \indexSet_i$
is fulfilled. So the index $k$ of every node $\xb_k$ is in exactly one $\indexSet_i$.
With these index sets we now define several submatrices and subvectors as follows
\begin{align*}
    \FEMatrix_{ij} &\defas \left( \FEMatrixEntry_{kl} \right)_{k \in \indexSet_i, l \in \indexSet_j}, \FEMatrix_{i\Ic} \defas \left( \FEMatrixEntry_{kl} \right)_{k \in \indexSet_i, K_{\nlDom} + 1 \leq l \leq K},\\ 
    \FEForce_i &\defas \left( \FEForceEntry_k^h \right)_{k \in \indexSet_i} \text{ and } \FEu_i \defas \left( \FEu_k^h \right)_{k \in \indexSet_i}.
\end{align*}
Then, we can rewrite \eqref{FE_weak_formulation} as
\begin{align*}
    \sum_{j=1}^n \FEMatrix_{ij} \FEu_j = \FEForce_i - \FEMatrix_{i\Ic}\FEDirData \quad \text{for all } i=1,...,K_\nlDom,.
\end{align*}
Furthermore, in the Schwarz method below we denote by $\FEu_i^k$ the solution for the subproblem on the domain $\nlDom_i$ in iteration $k$. The formulation of the additive Schwarz method is shown in Algorithm \ref{alg:FE_additive_schwarz}. In order to compute solution $\FEu_i^{k+1}$ the solutions $\FEu_j^k$ regarding the other domains are known and serve as Dirichlet boundary conditions. In contrast to \eqref{FE_weak_formulation} they are put on the right-hand side of the equation. In this case the algorithm is equivalent to a block-Jacobi method to solve the linear system \eqref{FE_weak_formulation}. \\
\begin{algorithm}[H]
\label{alg:FE_additive_schwarz}
\SetAlgoLined
\DontPrintSemicolon
\SetKwInOut{Input}{input}
\Input{$\FEu_i^0$ for $i=1,...,n$}
\For{$k=0,1,...$ until convergence}{
\For{$i=1,...,n$}{
Find $\FEu_i^{k+1}$ s.t.
\begin{align} \label{FE_additive_step}
\FEMatrix_{ii} \FEu_i^{k+1} = \FEForce_i - \FEMatrix_{i\Ic}\FEDirData - \sum_{j \neq i} \FEMatrix_{ij}\FEu_j^k
\end{align}
}
}
\caption{Additive Schwarz Method (Block-Jacobi Method)}
\end{algorithm}~\\
The description of the multiplicative Schwarz method is depicted in Algorithm \ref{alg:FE_multiplicative_schwarz}. As mentioned above, in the case of the multiplicative Schwarz formulation the most recent solution $\FEu_j^{k+1}$ for $j=1,...,i-1$ is used to solve the problem regarding $\nlDom_i$. This alternating Schwarz method is equivalent to the block-Gauß-Seidel method applied on the linear system \eqref{FE_weak_formulation}.\\
\begin{algorithm}[H]
\label{alg:FE_multiplicative_schwarz}
\SetAlgoLined
\DontPrintSemicolon
\SetKwInOut{Input}{input}
\Input{$\FEu_i^0$ for $i=1,...,n$}
\For{$k=0,1,...$ until convergence}{
\For{$i=1,...,n$}{
Find $\FEu_i^{k+1}$ s.t.
\begin{align*}
\FEMatrix_{ii} \FEu_i^{k+1} = \FEForce_i - \FEMatrix_{i\Ic}\FEDirData - \sum_{j = 1}^{i-1} \FEMatrix_{ij}\FEu_j^{k+1} - \sum_{j = i+1}^{n} \FEMatrix_{ij}\FEu_j^k
\end{align*}
}
}
\caption{Multiplicative Schwarz Method \\ (Block-Gauss-Seidel Method)}
\end{algorithm}~\\
The description and analysis of the block-Jacobi and block-Gauß-Seidel method for linear systems can be found in, e.g., \cite[Chapter 3]{hackbusch} or \cite[Chapter 4]{saad}.
\begin{remark}
    \label{remark:conv_discrete_schwarz}
    We recall from Remark \ref{remark:pos_def_stiff_matrix} that we only consider nonlocal problems, where the matrix $\FEMatrix$ is symmetric and positive definite. Then, the multiplicative Schwarz method converges linearly (see \cite[Theorem 3.53 and 3.39]{hackbusch}) and, if $n=2$, i.e., the domain $\nlDom$ is decomposed in two nonoverlapping subdomains $\nlDom_1$ and $\nlDom_2$, the additive Schwarz method is also converging in a linear fashion to the unique solution, which follows from \cite[Corollary 3.52 and Theorem 3.36]{hackbusch}. However, if $\nlDom$ is decomposed in more than two subdomains, the linear convergence of a relaxed block-Jacobi algorithm as defined in \cite[Chapter 12.5.3]{hackbusch} can be shown, if the dampening or Richardson parameter $\theta$ is small enough (see \cite[Theorem 14.5 and 14.6]{saad}).
\end{remark}
\section{Patch Tests}
\label{chap:patch_tests}
In this section we show that Schwarz methods for nonlocal Dirichlet problems, where we employ nonlocal diffusion operators, can be interpreted as a Nonlocal-to-Nonlocal coupling method, that satisfies linear, quadratic and cubic patch tests.
Analogously to \cite[Definition 1 and 2]{coupling_review}, we define the linear and higher-order patch test for the coupling of two nonlocal operators as follows:
\begin{definition}
Given two nonlocal operators $-\nlOp_1$ and $-\nlOp_2$ and a linear function \\
$u^*(\xb) \defas c_0 + c_1 \xb$, where $c_0,c_1 \in \mbRd$ and $u^*$ is a solution of 
\begin{align}
 -\nlOp_i u = 0 \text{ on } \nlDom,\quad u(\xb) = c_0 + c_1 \xb \text{ on } \nlBound \text{ for } i=1,2.  \label{patch_test_linear}
\end{align}
Then, a coupling method passes the linear patch test, if $u^*$ is also a solution of the coupled problem with the same boundary condition.
\end{definition}
\begin{definition}
Given two nonlocal operators $-\nlOp_1$ and $-\nlOp_2$, and a polynomial $u^*(\xb) \defas \sum_{|\multiIndex|=0}^{p} c_{\multiIndex} \xb^{\multiIndex}$ with degree $p \in \{2,3\}$, multi-indices $\multiIndex \in \Nbb_0^{d}$ and $c_{\multiIndex} \in \mbR$. Additionally, suppose that $u^*$ is a solution of 
\begin{align}
 -\nlOp_i u = f^{poly} \text{ on } \nlDom,\quad u(\xb) = \sum_{|\multiIndex|=0}^{p} c_{\multiIndex} \xb^{\multiIndex} \text{ on } \nlBound \text{ for } i=1,2, \label{patch_test_quadratic_cubic} 
\end{align}
where $f^{poly}$ is a polynomial with degree $p-2$. Then, a coupling method passes the quadratic(cubic) patch test, if $p=2$ ($p=3$) and $u^*$ is also a solution of the coupled problem with the same boundary condition.
\end{definition}
Given a solution $u^*$ of a nonlocal Dirichlet problem \eqref{patch_test_linear} or \eqref{patch_test_quadratic_cubic}, and the right-hand side $f^{poly}$ of the corresponding linear, quadratic or cubic patch test, we can directly follow that $u^*$ is also the solution to the Schwarz formulation \eqref{schwarz_formulation} since the Dirichlet boundary conditions on every $\nlBound_i$ holds and
\begin{align*}
    - \nlOp_i u^* = f^{poly} \text{ on } \nlDom_i \text{ for } i=1,...,n.
\end{align*}
Therefore the patch test is trivially fulfilled. Thus, we can also interpret the Schwarz method as a coupling method. As we will demonstrate in Section \ref{chap:numerical_experiments}, we can, e.g., couple a singular and a constant kernel.\\
Next, we present a class of kernels where the corresponding nonlocal operator is equivalent to the Laplace operator $- \Delta u(\xb) \defas \sum_{i=1}^d \partial_i^2 u(\xb)$ in case of polynomials up to degree three $\polynom(\xb) = \sum_{|\multiIndex|=0}^3 c_{\multiIndex}\xb^{\multiIndex}$.
Hence, given a radially symmetric kernel $\kernel_{\delta}(\xb,\yb) \defas \kernel_{\delta}(|\yb - \xb|)\ind_{B_{\delta}(\xb)}(\yb)$ that satisfies 
\begin{align}\label{requirements_kernel_patch_test}
\int_{B_{\delta}(0)} \zb_i^2 \kernel_{\delta}(|\zb|) ~d\zb = 2,
\end{align}
we derive for such polynomials $\polynom$ that
\begin{align*}
    -\nlOp \polynom(\xb) &= \int_{B_{\delta}(\xb)} \left(\polynom(\xb) - \polynom(\yb) \right)\kernel_{\delta}(\xb,\yb) ~d\yb  = \int_{B_{\delta}(0)} \left(\polynom(\xb) - \polynom(\xb + \zb) \right)\kernel_{\delta}(|\zb|) ~d\yb \\
    &= - \int_{B_{\delta}(0)} \sum_{\multiIndex=1}^3 \frac{D^{\multiIndex}\polynom(\xb)}{|\multiIndex|!}\zb^{\multiIndex} \kernel_{\delta}(|\zb|) ~d\zb
    = - \int_{B_{\delta}(0)} \frac{1}{2} \sum_{i=1}^d \partial_i^2 \polynom(\xb)\zb_i^{2} \kernel_{\delta}(|\zb|) ~d\zb = - \Delta \polynom(\xb),
\end{align*}
where we used in the last step that the function $\zb^{\multiIndex}$ is an odd function regarding one $x_i$-axis, if $|\multiIndex|$ is an odd number. Furthermore, if $\multiIndex = (\multiIndex_1,...,\multiIndex_d)$ with $|\multiIndex|=2$, $\multiIndex_i = \multiIndex_j =1$ and $i \neq j$, the function $\zb^{\multiIndex}= \zb_i \zb_j$ is an odd function regarding the $x_j$-axis. Since $B_{\delta}(0)$ is a symmetric domain regarding every $x_i$-axis, the integrals, in which these functions $\zb^{\multiIndex}$ appear, vanish. In Section \ref{chap:patch_test_experiments} we use different kernels that satisfy the conditions \eqref{requirements_kernel_patch_test} in a numerical example, where the nonlocal Schwarz formulation passes the linear patch test.
\section{Numerical Experiments} \label{chap:numerical_experiments}
In the following we numerically examine the Schwarz method for several nonlocal problems. We start with an example for the nonlocal Dirchlet problem and continue with a Neumann problem. After that we investigate the patch test for one example. In the last section we study two preconditioned GMRES versions that we compare with GMRES without preconditioner. We apply the finite element method in all experiments as mentioned in Chapter \ref{chap:finite_element_method}. Moreover, we consider the residual error regarding the Euclidean norm $||\FEMatrix  \FEu^k - \FEForceAlt||_2$, where $\FEu^k$ is the solution after the $k$-th iteration and ${\FEForceAlt \defas \FEForce - \FEMatrix_{\nlBound}\FEDirData}$. In our experiments the energy error ${||\FEu^{k}-\FEu^{k-1}||_{\FEMatrix} \defas \left(\FEu^{k} - \FEu^{k-1} \right)^\top \FEMatrix \left(\FEu^{k} - \FEu^{k-1} \right)}$ had the same behaviour as the residual error. For a more concise presentation, we therefore only discuss results involving the residual error.
\subsection{Nonlocal Dirichlet Problem}
\label{chap:dirichlet_experiments}
In our first experiment, we compute a nonlocal Dirichlet problem as described in Section \ref{sec:dirichlet_problem}. Here, we choose $\nlDom$ as depicted in Figure \ref{fig:dirichlet_problem} and $\delta = 0.1$. Thus, we have $\bar{\nlDom} = \bar{\nlDom}_1 \cup \bar{\nlDom}_2 \cup \bar{\nlDom}_3$. 
\begin{figure}[h!]
	\centering
	\def\svgwidth{0.4\textwidth}
	{
\begingroup%
  \makeatletter%
  \providecommand\color[2][]{%
    \errmessage{(Inkscape) Color is used for the text in Inkscape, but the package 'color.sty' is not loaded}%
    \renewcommand\color[2][]{}%
  }%
  \providecommand\transparent[1]{%
    \errmessage{(Inkscape) Transparency is used (non-zero) for the text in Inkscape, but the package 'transparent.sty' is not loaded}%
    \renewcommand\transparent[1]{}%
  }%
  \providecommand\rotatebox[2]{#2}%
  \newcommand*\fsize{\dimexpr\f@size pt\relax}%
  \newcommand*\lineheight[1]{\fontsize{\fsize}{#1\fsize}\selectfont}%
  \ifx\svgwidth\undefined%
    \setlength{\unitlength}{540bp}%
    \ifx\svgscale\undefined%
      \relax%
    \else%
      \setlength{\unitlength}{\unitlength * \real{\svgscale}}%
    \fi%
  \else%
    \setlength{\unitlength}{\svgwidth}%
  \fi%
  \global\let\svgwidth\undefined%
  \global\let\svgscale\undefined%
  \makeatother%
  \begin{picture}(1,1.00185182)%
    \lineheight{1}%
    \setlength\tabcolsep{0pt}%
    \put(0,0){\includegraphics[width=\unitlength,page=1]{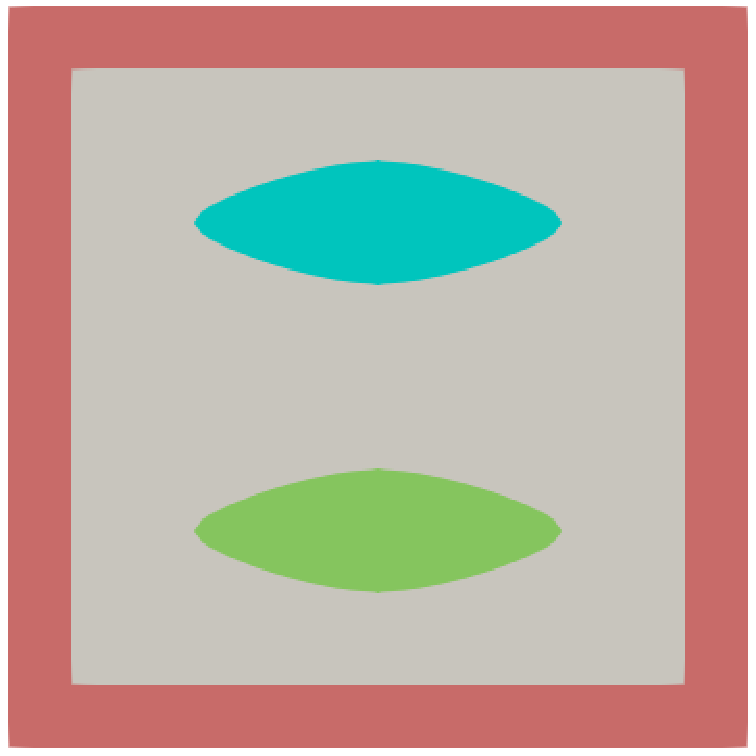}}%
    \put(0.45,0.69578661){\color[rgb]{0,0,0}\makebox(0,0)[lt]{\lineheight{1.25}\smash{\begin{tabular}[t]{l}$\nlDom_1$\end{tabular}}}}%
    \put(0.45,0.28139052){\color[rgb]{0,0,0}\makebox(0,0)[lt]{\lineheight{1.25}\smash{\begin{tabular}[t]{l}$\nlDom_2$\end{tabular}}}}%
    \put(0.8,0.475){\color[rgb]{0,0,0}\makebox(0,0)[lt]{\lineheight{1.25}\smash{\begin{tabular}[t]{l}$\nlDom_3$\end{tabular}}}}%
    \put(0.39606574,0.93){\color[rgb]{0,0,0}\makebox(0,0)[lt]{\lineheight{1.25}\smash{\begin{tabular}[t]{l}$\nlBound$\end{tabular}}}}%
  \end{picture}%
\endgroup%
}
	\caption{The domain $\nlDom = (0,1)^2$ is divided into $\nlDom_1$, $\nlDom_2$ and $\nlDom_3$. The nonlocal boundary $\nlBound$ is depicted in red.}
	\label{fig:dirichlet_problem}
\end{figure}
Further, we use the following kernel, which satisfies the requirements of a singular symmetric kernel of Remark \ref{remark:examples_well_posedness}, and we employ the subsequent piecewise constant forcing term and boundary data
\begin{align*}
	\kernel(\xb,\yb) &= \begin{cases}
		\frac{4 c_\delta}{||\xb - \yb||_2^{2 + 2s}} \ind_{B_\delta(\xb)}(\yb) & \xb,\yb \in \nlDom_1 \text{ or }\xb,\yb \in \nlDom_2, \\
		\frac{7 c_\delta}{||\xb - \yb||_2^{2 + 2s}} \ind_{B_\delta(\xb)}(\yb) & \xb,\yb \in \nlDom_3, \\
		\frac{5 c_\delta}{||\xb - \yb||_2^{2 + 2s}} \ind_{B_\delta(\xb)}(\yb) & \xb \in \nlBound \text{ or } \yb \in \nlBound, \\
		\frac{10 c_\delta}{||\xb - \yb||_2^{2 + 2s}} \ind_{B_\delta(\xb)}(\yb) & \text{else},
	\end{cases} \\
	c_\delta &= \frac{2-2s}{\pi \delta^{2-2s}},\ s=0.5,\ 
	f(\xb) = 5\ind_{\nlDom_1 \cup \nlDom_2}(\xb) + 1\ind_{\nlDom_3}(\xb) \text{ and } g = 0 \text{ on } \nlBound.
\end{align*}
Here, the kernel $\kernel$ can, e.g., be considered as a special case of the kernels investigated in \cite{delia_glusa_interface} or as a Case 1 kernel of \cite{DuAnalysis}.
The stiffness matrices for every subproblem are computed by using the python package nlfem\cite{nlfem}, which has to be done only once at the beginning of the algorithm.
Then, every subproblem is solved with LGMRES with tolerance $10^{-12}$ and the Schwarz method stops, when the residual error is below $10^{-9}$. The results for the multiplicative and additive Schwarz method are shown in Figure \ref{fig:dirichlet_results} for different choices of the mesh parameter $h$. Additionally, we can observe quadratic $h$-convergence(see Figure \ref{fig:h_convergence}).
\begin{figure}[h!]
    \centering
	\includegraphics[width=0.7\textwidth]{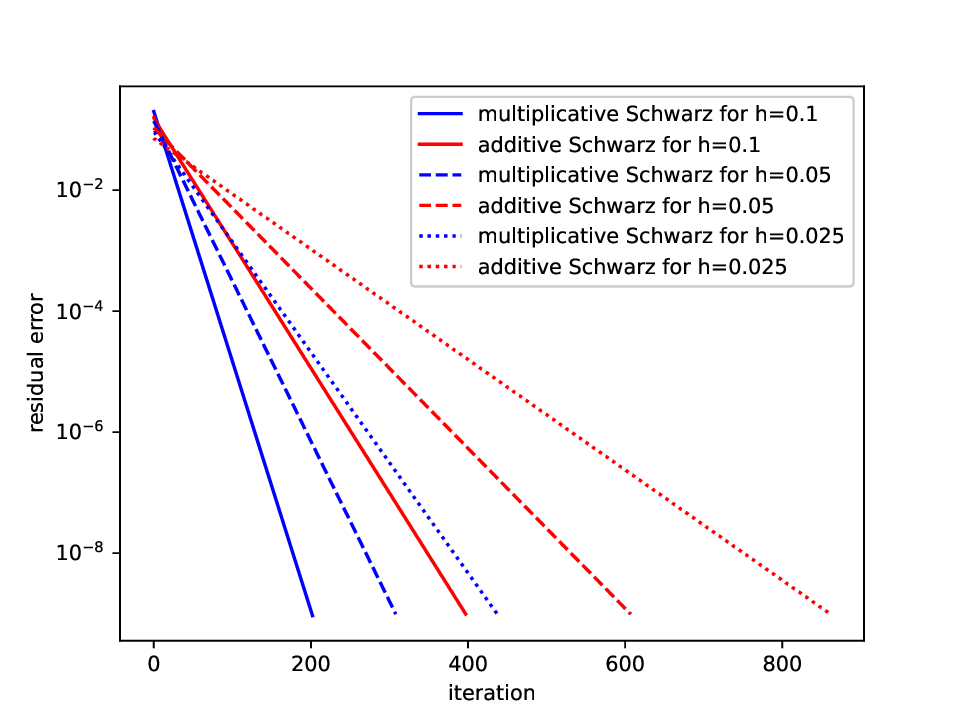}
	\caption{Here, we can see the residual error of the multiplicative and additive Schwarz method for the nonlocal Dirichlet problem with a singular symmetric kernel regarding different choices for the mesh resolution $h$. In all cases the error decreases in a linear fashion, which we expected at least for the multiplicative Schwarz version due to Remark \ref{remark:conv_discrete_schwarz}. Moreover, the additive version needs roughly twice as many iterations as the multiplicative Schwarz algorithm.}
\label{fig:dirichlet_results}
\end{figure}
\newpage
\subsection{Nonlocal Problem with Neumann Boundary}
\label{chap:neumann_experiments}
In this section we define a nonlocal problem with Neumann boundary conditions as described in \ref{sec:Neumann}, where we can have both Neumann and Dirichlet boundary conditions for a subproblem, see \eqref{schwarz_formulation_neumann}. The setting is shown in Figure \ref{fig:neumann}. 
By applying the finite element method, we also have degrees of freedom on the nonlocal boundary $\nlBound$, which we therefore divide into $\nlBound_1^{N}$ and $\nlBound_2^{N}$.
Moreover, we choose an integrable kernel $\kernel$ and $f$ as well as $\kappa$ to be piecewise constant and dependent on $\nlDom_1$ and $\nlDom_2$ as follows
\begin{align*}
	\nlDom &= \left(0, 1\right)^2,\ \delta = 0.1,\\
	\kernel(\xb, \yb) &= 
	\frac{4}{\pi \delta^4}\ind_{B_\delta(\xb)}(\yb),\\
	\kappa(\xb) &= 1 \ind_{\nlDom_1}(\xb) + 10\ind_{\nlDom_2}(\xb) \text{ and }\\
	f(\xb) &= 10 \ind_{\nlDom_1}(\xb) + 1\ind_{\Omega_2}(\xb).
\end{align*}
For these tests we only consider the multiplicative Schwarz approach(see Remark \ref{remark:add_schwarz_two_domains}), and we take the same tolerances as in the previous Chapter, i.e., tolerance $10^{-12}$ for LGMRES and $10^{-9}$ regarding the residual error as the termination criterion for the multiplicative Schwarz algorithm.
As we can observe in Figure \ref{fig:neumann_errors} and as expected in Remark \ref{remark:conv_discrete_schwarz}, the behavior of the errors corresponding to the multiplicative Schwarz method stay the same compared to Schwarz methods for nonlocal Dirichlet problems, i.e., the residual error decreases in a linear way. However, the number of needed iterations to fulfill the termination criterion stays roughly the same, which can be explained by the fact, that the condition number of the stiffness matrix in the continuous Galerkin approach(without splitting of inner boundary basis functions) is constant for integrable kernels(see \cite[Theorem 6.3]{DuAnalysis}). Lastly, we also notice quadratic $h$-convergence for this example, which is illustrated in Figure \ref{fig:h_convergence}.
\begin{figure}[h!]
	\centering
	\def\svgwidth{0.375\textwidth}
	{ 
\begingroup%
  \makeatletter%
  \providecommand\color[2][]{%
    \errmessage{(Inkscape) Color is used for the text in Inkscape, but the package 'color.sty' is not loaded}%
    \renewcommand\color[2][]{}%
  }%
  \providecommand\transparent[1]{%
    \errmessage{(Inkscape) Transparency is used (non-zero) for the text in Inkscape, but the package 'transparent.sty' is not loaded}%
    \renewcommand\transparent[1]{}%
  }%
  \providecommand\rotatebox[2]{#2}%
  \newcommand*\fsize{\dimexpr\f@size pt\relax}%
  \newcommand*\lineheight[1]{\fontsize{\fsize}{#1\fsize}\selectfont}%
  \ifx\svgwidth\undefined%
    \setlength{\unitlength}{595.27559055bp}%
    \ifx\svgscale\undefined%
      \relax%
    \else%
      \setlength{\unitlength}{\unitlength * \real{\svgscale}}%
    \fi%
  \else%
    \setlength{\unitlength}{\svgwidth}%
  \fi%
  \global\let\svgwidth\undefined%
  \global\let\svgscale\undefined%
  \makeatother%
  \begin{picture}(1,1.41428571)%
    \lineheight{1}%
    \setlength\tabcolsep{0pt}%
    \put(0,0){\includegraphics[width=\unitlength]{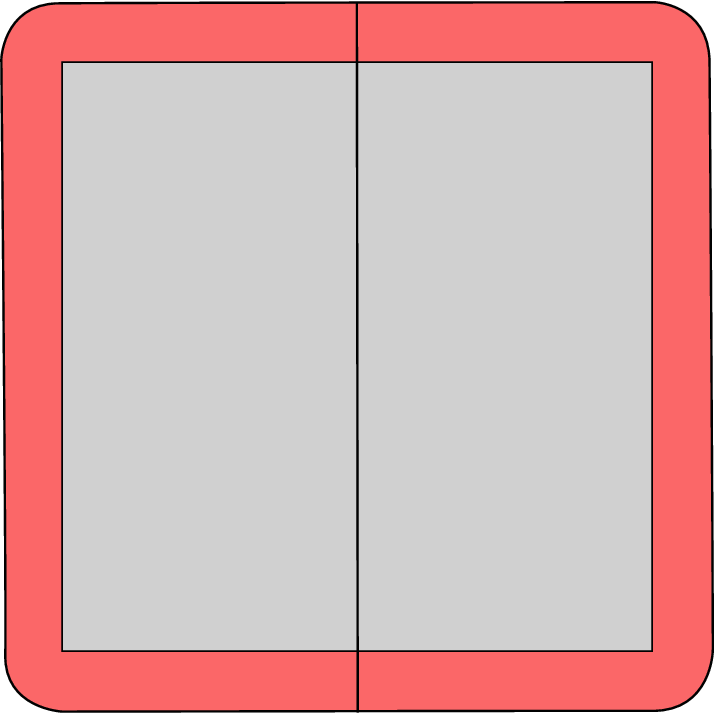}}%
    \put(0.25,0.45){\color[rgb]{0.02352941,0.01960784,0.01960784}\makebox(0,0)[lt]{\lineheight{1.25}\smash{\begin{tabular}[t]{l}$\nlDom_1$\end{tabular}}}}%
    \put(0.7,0.45){\color[rgb]{0.02352941,0.01960784,0.01960784}\makebox(0,0)[lt]{\lineheight{1.25}\smash{\begin{tabular}[t]{l}$\nlDom_2$\end{tabular}}}}%
    \put(0.7,0.025){\color[rgb]{0.02352941,0.01960784,0.01960784}\makebox(0,0)[lt]{\lineheight{1.25}\smash{\begin{tabular}[t]{l}$\nlBound_2^N$\end{tabular}}}}%
    \put(0.25,0.935){\color[rgb]{0.02352941,0.01960784,0.01960784}\makebox(0,0)[lt]{\lineheight{1.25}\smash{\begin{tabular}[t]{l}$\nlBound_1^N$\end{tabular}}}}%
  \end{picture}%
\endgroup%
}
	\caption{In this example $\nlDom$ is divided in two subdomains $\nlDom_1$ and $\nlDom_2$ and the nonlocal boundary $\nlBound$ is decomposed in $\nlBound_1^N$ and $\nlBound_2^N$.}
	\label{fig:neumann}
\end{figure}
\begin{figure}[h!]
    \centering
	\includegraphics[width=0.7\textwidth]{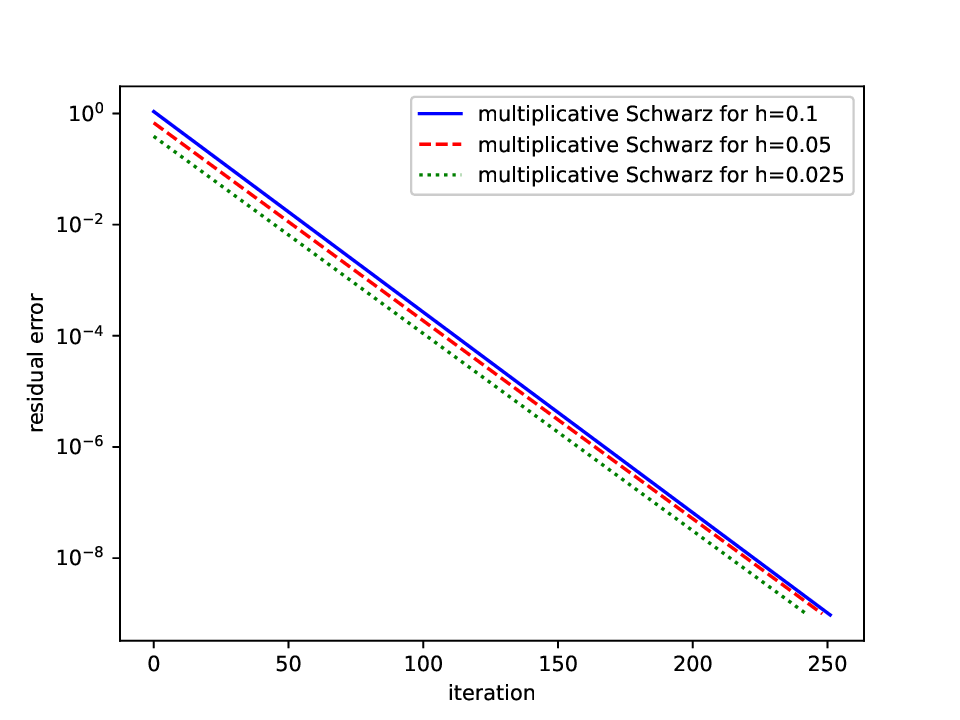}
	\caption{In this picture the residual error regarding the multiplicative Schwarz method for the nonlocal Problem with Neumann boundary condition w.r.t. a selection of mesh parameter $h$ is depicted. Again, we can observe a linear decrease in the residual error.}
	\label{fig:neumann_errors}
\end{figure}
\newpage
\subsection{Patch Test}
\label{chap:patch_test_experiments}
\begin{figure}
	\centering
	\def\svgwidth{0.4\textwidth}
	{\small 
\begingroup%
  \makeatletter%
  \providecommand\color[2][]{%
    \errmessage{(Inkscape) Color is used for the text in Inkscape, but the package 'color.sty' is not loaded}%
    \renewcommand\color[2][]{}%
  }%
  \providecommand\transparent[1]{%
    \errmessage{(Inkscape) Transparency is used (non-zero) for the text in Inkscape, but the package 'transparent.sty' is not loaded}%
    \renewcommand\transparent[1]{}%
  }%
  \providecommand\rotatebox[2]{#2}%
  \newcommand*\fsize{\dimexpr\f@size pt\relax}%
  \newcommand*\lineheight[1]{\fontsize{\fsize}{#1\fsize}\selectfont}%
  \ifx\svgwidth\undefined%
    \setlength{\unitlength}{560.00001526bp}%
    \ifx\svgscale\undefined%
      \relax%
    \else%
      \setlength{\unitlength}{\unitlength * \real{\svgscale}}%
    \fi%
  \else%
    \setlength{\unitlength}{\svgwidth}%
  \fi%
  \global\let\svgwidth\undefined%
  \global\let\svgscale\undefined%
  \makeatother%
  \begin{picture}(1,1)%
    \lineheight{1}%
    \setlength\tabcolsep{0pt}%
    \put(0,0){\includegraphics[width=\unitlength]{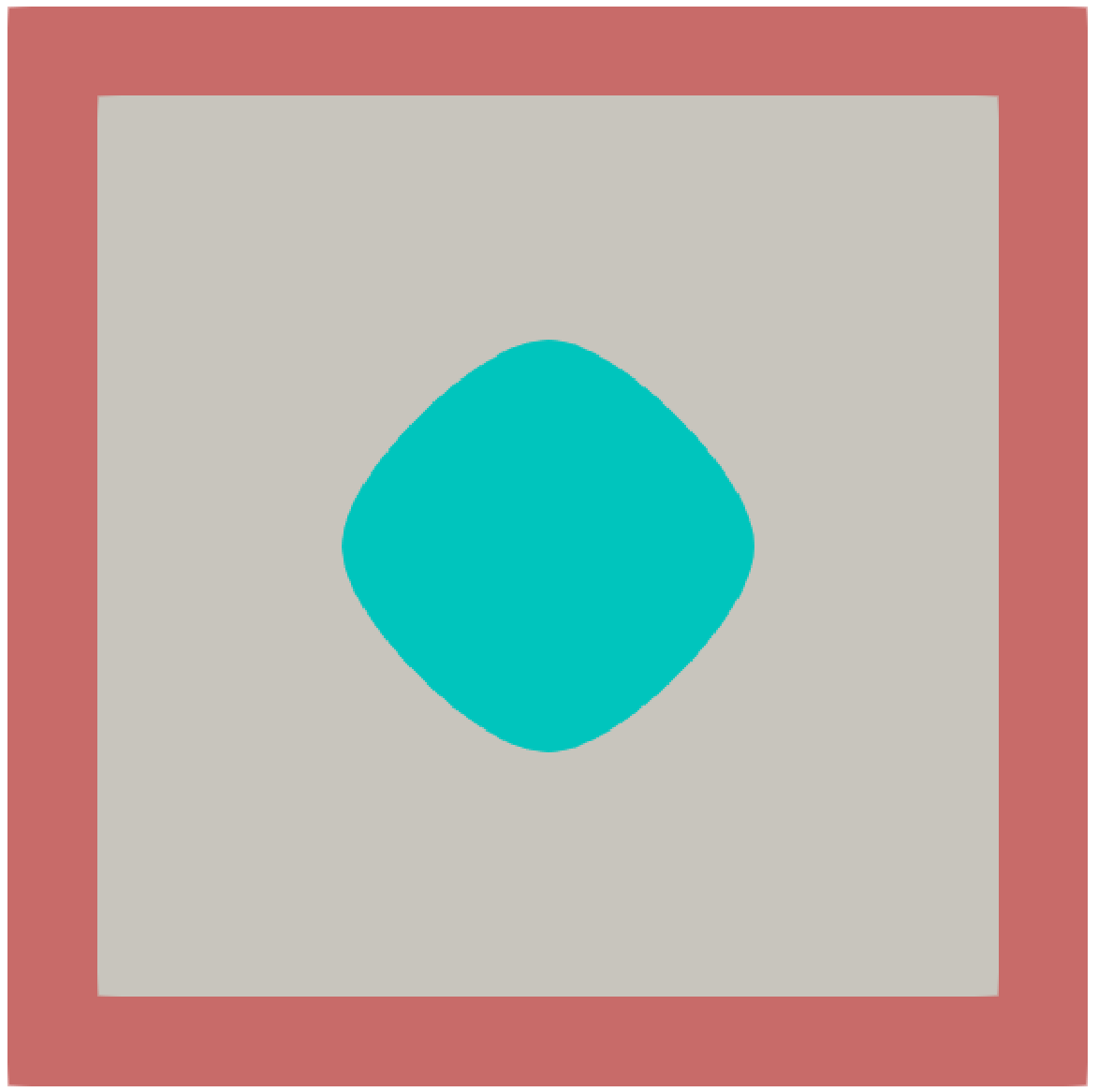}}%
    \put(0.475,0.475){\color[rgb]{0.0,0.0,0.0}\makebox(0,0)[lt]{\lineheight{1.25}\smash{\begin{tabular}[t]{l}$\nlDom_1$\end{tabular}}}}%
    \put(0.68670533,0.19740927){\color[rgb]{0,0,0}\makebox(0,0)[lt]{\lineheight{1.25}\smash{\begin{tabular}[t]{l}$\nlDom_2$\end{tabular}}}}%
    \put(0.24810692,0.93){\color[rgb]{0.0,0.0,0.0}\makebox(0,0)[lt]{\lineheight{1.25}\smash{\begin{tabular}[t]{l}$\nlBound$\end{tabular}}}}%
  \end{picture}%
\endgroup%
}
	\caption{Here, we see the decomposition of $\nlDom$ in a turquoise area $\nlDom_1$ and in a gray domain $\nlDom_2$ that we use for the patch test and for testing the (preconditioned) GMRES in Chapter \ref{chap:preconditioned_gmres}. The nonlocal boundary $\nlBound$ is again depicted in red.}
	\label{fig:patch_test_domains}
\end{figure}
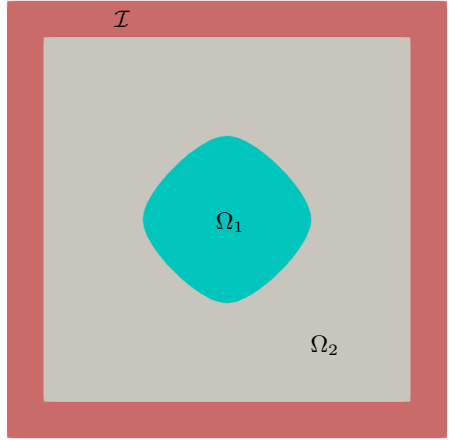
Now, we conduct a linear patch regarding two nonlocal diffusion operators. We solve in our case
\begin{alignat*}{4}
	-\nlOp_1 \weakSol_1 &= 0\quad\quad &&\text{on } \nlDom_1, & -\nlOp_2 \weakSol_2 &= 0 \quad\quad&&\text{on } \nlDom_2, \\
	\weakSol_1 &= \weakSol_2 &&\text{on } \nlDom_2, \hspace{3em}\text{and}\hspace{3em}& \weakSol_2 &= \weakSol_1 &&\text{on } \nlDom_1, \\
	\weakSol_1 &= g &&\text{on } \nlBound & \weakSol_2 &= g &&\text{on } \nlBound,
\end{alignat*} where
\begin{align*}
	g(\xb) &\defas \xb_1 + \xb_2 \text{ for } \xb \in \completeDom,\ -\nlOp_i \weakSol(\xb) \defas \int_{\completeDom} \left(\weakSol(\xb) - \weakSol(\yb)\right) \kernel_i(\xb,\yb) ~d\yb \text{ with}\\
\kernel_1 (\xb, \yb) &\defas \frac{4}{\pi \delta^4} \ind_{B_{\delta}(\xb)}(\yb),\ \kernel_2 (\xb, \yb) \defas \frac{c_{\delta}}{||\xb - \yb||^{2 + 2s}} \ind_{B_{\delta}(\xb)}(\yb), ~s=0.6 \text{ and } c_{\delta} \defas \frac{2 - 2s}{\pi \delta^{2-2s}}. 
\end{align*}
Here, both kernels $\kernel_1$ and $\kernel_2$ are symmetric, but the resulting kernel $\kernel$ on the complete domain $\left(\completeDom\right) \times \left(\completeDom\right)$ with
\begin{align*}
    \kernel(\xb, \yb) = \kernel_1(\xb, \yb) \ind_{\nlDom_1}(\xb) + \kernel_2(\xb, \yb) \ind_{\nlDom_2}(\xb)
\end{align*}
is nonsymmetric.
In Figure \ref{fig:patch_test_domains} the setup is the illustrated and the starting and final solution are depicted in Figure \ref{fig:patch_test_solutions}. We again used LGMRES with tolerance $10^{-12}$ and stopped after the residual error dropped under $10^{-9}$.
\begin{figure}
	\centering
	\includegraphics[width=0.4\textwidth]{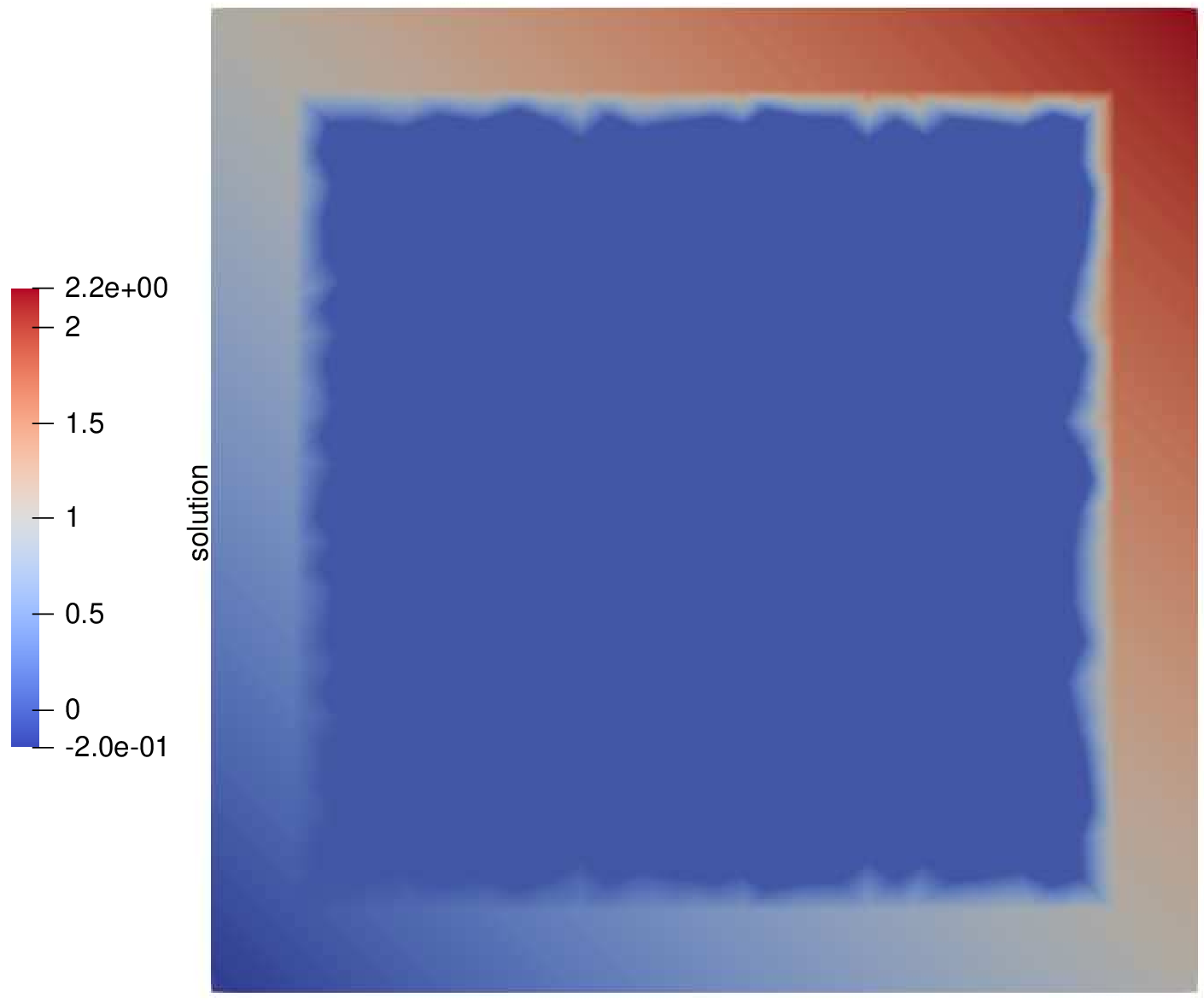}\hspace{5em}
	\includegraphics[width=0.4\textwidth]{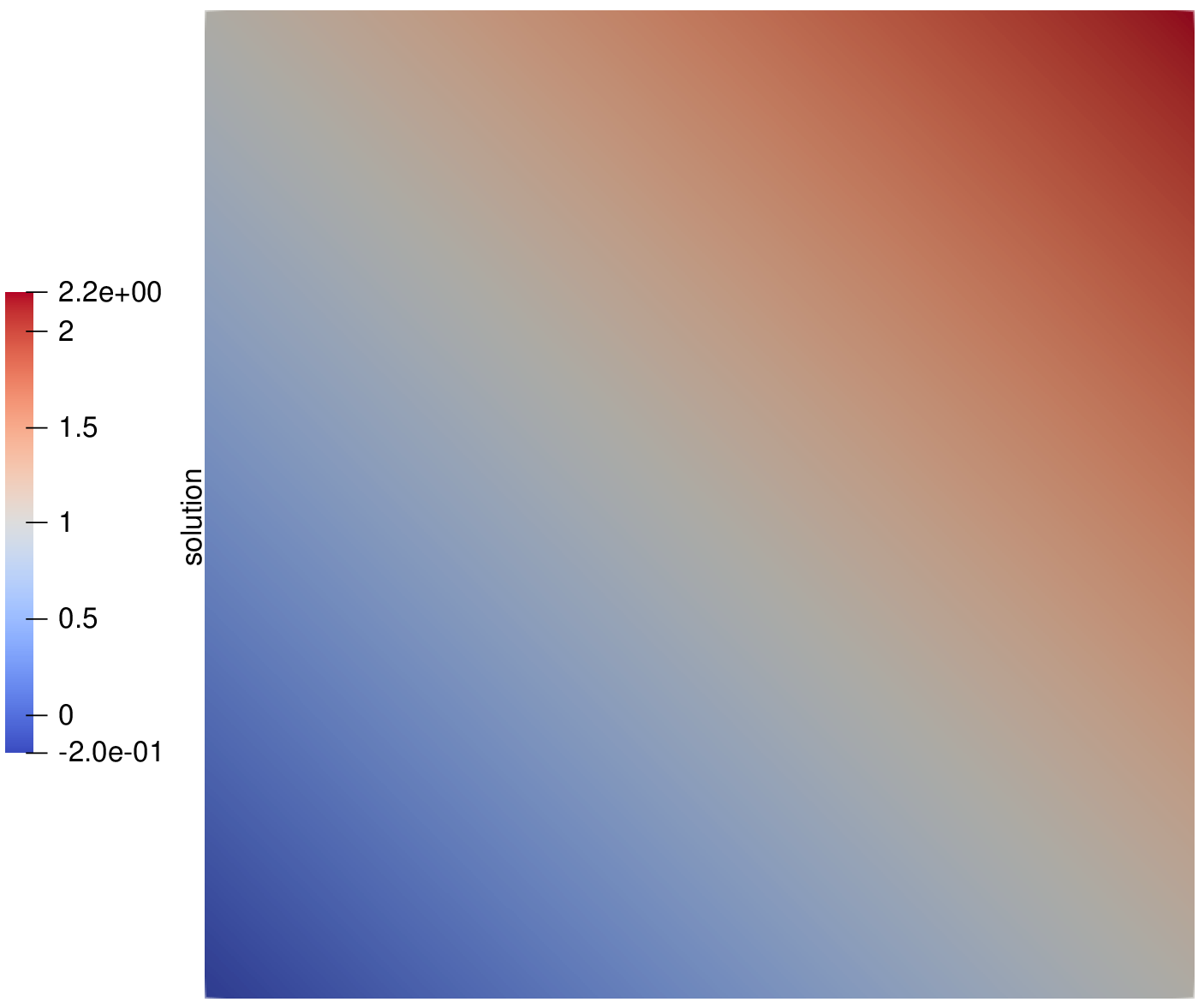}
	\caption{On the left-hand side we see the starting solution, where we have $\weakSol^0 = g$ on the nonlocal boundary $\nlBound$ and $\weakSol^0 = 0$ on the domain $\nlDom$. Additionally, the final solution is shown on the right-hand side.}
	\label{fig:patch_test_solutions}
\end{figure}
The convergence results w.r.t. the residual error can be seen in Figure \ref{fig:patch_test_convergence}, where we also observe linear convergence.
Since we know the exact solution $g$, we can examine the $L^2(\completeDom)$-norm distance of several solutions $\weakSol_h^p$ of the patch test, that depend on the choice of the mesh parameter $h$, to the the function $g_{0.01}$, which is the projection of $g$ onto the mesh with size $h=0.01$. The results are presented in Figure \ref{fig:h_convergence}. In this case, we again observe quadratic $h$-convergence.
\begin{figure}
	\centering
	\includegraphics[width=0.7\textwidth]{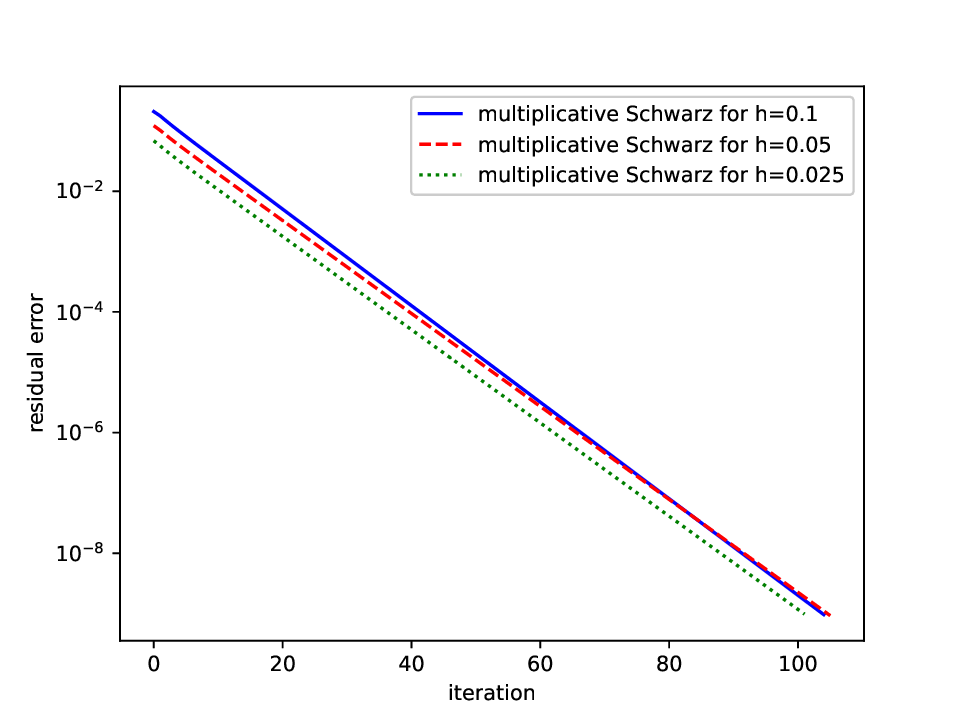}
	\caption{In case of this linear patch test, we can observe linear convergence regarding the residual error for the multiplicative Schwarz approach.}
	\label{fig:patch_test_convergence}
\end{figure}
\begin{figure}
    \centering
	\includegraphics[width=0.7\textwidth]{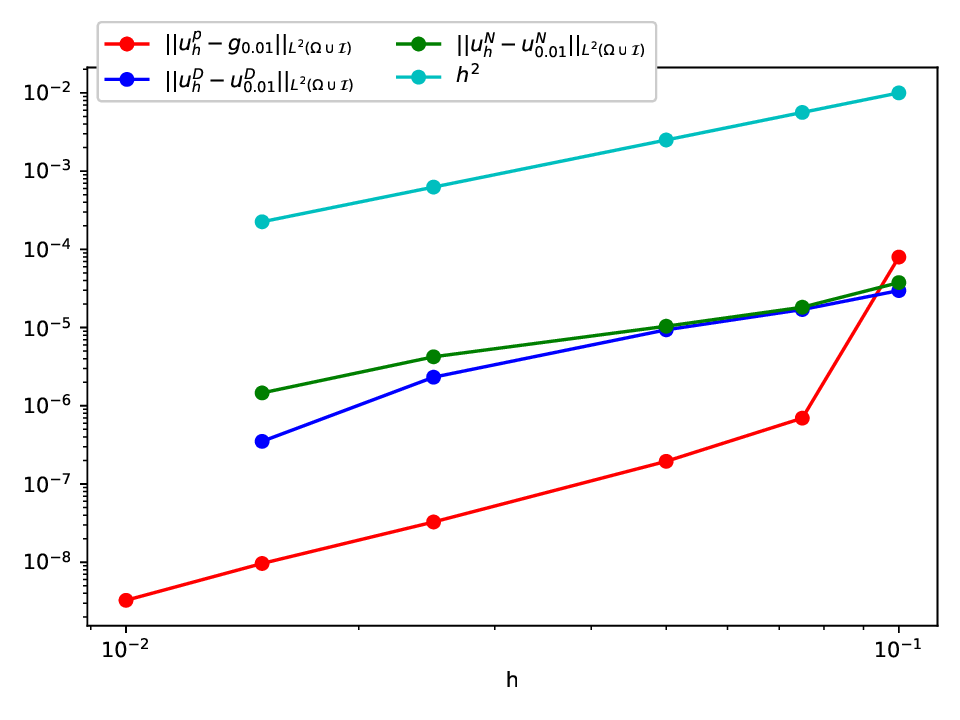}
	\caption{Here, we denote by $\weakSol_h^D$, $\weakSol_h^N$ and $\weakSol_h^p$ the final solution of the Dirichlet problem, the Neumann problem and the patch test in dependence of the mesh parameter $h$ as described in Chapters \ref{chap:dirichlet_experiments}, \ref{chap:neumann_experiments} and \ref{chap:patch_test_experiments}. In this Figure, we compare $\weakSol_h^D$ and $\weakSol_h^N$ in the $L^2(\completeDom)$-norm to the solution $\weakSol^D_{0.01}$ and $\weakSol_{0.01}^N$, respectively, where $h=0.01$. Moreover, in case of the patch test, we compute the $L^2(\completeDom)$-norm distance of $\weakSol_k^p$ to the projection of the exact solution $g$ onto the space of continuous piecewise linear basis functions regarding the mesh with size $h=0.01$, which is indicated by $g_{0.01}$. In all three cases, we can observe a quadratic convergence.}
	\label{fig:h_convergence}
\end{figure}
\subsection{Preconditioned GMRES}
\label{chap:preconditioned_gmres}
In the finite element setting the multiplicative Schwarz method is equivalent to the block-Gauß-Seidel algorithm and the additive Schwarz method corresponds to the block-Jacobi algorithm as seen in Section \ref{chap:schwarz_finite_elements}. Therefore, we make use of the block-Gauß-Seidel preconditioner $\precBGS^{-1}$ and the block-Jacobi preconditioner $\precBJ^{-1}$, that are defined as follows
\begin{align*}
	\precBGS = \begin{pmatrix}
		\FEMatrix_{11} & \cdots & 0      \\
		\vdots         & \ddots	& \vdots \\
		\FEMatrix_{1n} & \cdots & \FEMatrix_{nn}
	\end{pmatrix} \text{ and }
	\precBJ = diag(\FEMatrix_{11},...,\FEMatrix_{nn}),
\end{align*}
to precondition GMRES. For more information on these two preconditioner we refer to \cite[Chapter 4.1.2]{saad}. In the tests we compare GMRES equipped with a left preconditioner $\precBGS^{-1}$ or $\precBJ^{-1}$ to GMRES without any preconditioning.\\
In all cases we used the constant kernel $\kernel_1$ on $\nlDom_1$ and the fractional kernel $\kernel_2$ on $\nlDom_2$ of Section \ref{chap:patch_test_experiments}. Moreover, we set the boundary data $g = 0$ on $\nlBound$ and the forcing term $f = 10$ on $\nlDom$. In every test the convergence tolerance of GMRES is chosen to be $10^{-10}$. Additionally, we denote by $\kappa_{GMRES}$ the condition number of the finite element matrix $\FEMatrix$, by $\kappa_{GMRES+BJ}$ the condition number of $\precBJ^{-1}\FEMatrix$ and by $\kappa_{GMRES+BGS}$ the condition number of $\precBGS^{-1}\FEMatrix$, respectively.
\begin{table}[h!]
	\begin{center}
		\begin{tabular}{|c c c c c c c|} 
			\hline
			$h$ & GMRES & GMRES+BJ & GMRES+BGS  & $\kappa_{\text{GMRES}}$ & $\kappa_{\text{GMRES+BJ}}$ & $\kappa_{\text{GMRES+BGS}}$ \\ [0.5ex] 
			\hline\hline
			0.1 & 177 & 39 & 15 & 2715.41 & 39.79 & 18.85  \\ 
			\hline
			0.05 & 365 & 30 & 15 & 4649.59 & 42.46 & 20.15 \\
			\hline
			0.025 & 580 & 32 & 13 & 8110.49 & 45.72 & 20.53 \\
			\hline
		\end{tabular}
		\caption{For $\delta=0.1$ and $s=0.5$ we tested different mesh sizes $h$ and documented the number of required iterations as well as the condition number of the (preconditioned) matrices.}
		\label{table1}
	\end{center}
\end{table}
\begin{table}[h!]
	\begin{center}
		\begin{tabular}{|c c c c c c c|} 
			\hline
			$s$ & GMRES & GMRES+BJ & GMRES+BGS & $\kappa_{\text{GMRES}}$ & $\kappa_{\text{GMRES+BJ}}$ & $\kappa_{\text{GMRES+BGS}}$ \\ [0.5ex] 
			\hline\hline
			0.2 & 118 & 25 & 12 & 483.18 & 35.22 & 18.68 \\ 
			\hline
			0.5 & 580 & 32 & 13 & 8110.49 & 45.72 & 20.53 \\
			\hline
			0.8 & 2509 & 37 & 14 & 11687.19 & 59.13 & 22.18 \\
			\hline
		\end{tabular}
		\caption{In these experiments we set $\delta=0.1$ and the mesh size $h=0.025$ was chosen. Then, we tested different values for the parameter $s$ of the singular kernel and recorded the condition number of the (preconditoned) system matrix and the number of iterations that each version of GMRES needed.}
		\label{table2}
	\end{center}
\end{table}
\begin{table}[h!]
	\begin{center}
		\begin{tabular}{|c c c c c c c|} 
			\hline
			$\delta_2$ & GMRES & GMRES+BJ & GMRES+BGS & $\kappa_{\text{GMRES}}$ & $\kappa_{\text{GMRES+BJ}}$ & $\kappa_{\text{GMRES+BGS}}$ \\ [0.5ex] 
			\hline\hline
			0.1 & 580 & 32 & 13 & 8110.49 & 45.72 & 20.53 \\ 
			\hline
			0.05 & 1089 & 35 & 14 & 18862.67 & 57.83 & 22.24\\
			\hline
			0.025 & 1675 & 36 & 14 & 39848.76 & 63.35 & 21.58 \\
			\hline
		\end{tabular}
		\caption{In these tests we set the horizon $\delta_1=0.1$ for the constant kernel and varied the parameter $\delta_2$ of the singular kernel. Moreover, we chose $h=0.025$ and $s=0.5$ and noted the number of iterations that each GMRES variation required as well as the condition number of the accompanying (preconditioned) matrices.}
		\label{table3}
	\end{center}
\end{table}
\newpage~\\
As we can see in all tests of Tables \ref{table1} - \ref{table3} the number of required iterations and the condition number reduce significantly, if we use the block-Jacobi or block-Gauß-Seidel preconditioner instead of GMRES with no preconditioning. For the block-Gauß-Seidel preconditioner the number of iterations as well as the condition number even stay roughly the same.

 \section{Conclusion}
 We have shown in this paper how the multiplicative and additive Schwarz method can be applied to nonlocal Dirichlet problems and how they can be utilized to solve nonlocal problems with Neumann boundary condition. In the first case we showed the convergence of the multiplicative version of the Schwarz algorithm for two widely used classes of symmetric kernels and in the second case we only needed symmetric kernels to prove that the multiplicative Schwarz approach converges. 
 Additionally, coupling nonlocal operators with the Schwarz method fulfills trivially nonlocal patch tests. In the last section we provided examples for all discussed problems in this work. Additionally, we observed in the patch test experiments in Chapter \ref{chap:patch_test_experiments}, that the Schwarz approach can also work for nonsymmetric kernels in practice. Further, we investigated preconditioned GMRES variants that resulted from the multiplicative or additive Schwarz algorithm in the finite element setting, where we noticed, that especially the block-Gauß-Seidel preconditioned GMRES version only needs a small amount of iterations.

\section*{Acknowledgements}
This work has been supported by the German Research Foundation (DFG) within the Research Training Group 2126: 'Algorithmic Optimization'.

\bibliographystyle{plain}
\bibliography{literature.bib}
\end{document}